\documentclass[leqno,11pt]{article}

\RequirePackage[amsthm,amsmath]{imsart}

\usepackage[mathscr]{eucal}
\usepackage{graphicx}
\usepackage{latexsym}
\usepackage{amssymb}
\usepackage{psfrag}
\usepackage{epsfig}
\usepackage{enumerate}
\DeclareMathAlphabet{\mathpzc}{OT1}{pzc}{m}{it}

\include{psbox}

\usepackage{amsfonts}
\usepackage{graphicx}
\usepackage{amsfonts}
\usepackage{color}
\usepackage[usenames,dvipsnames]{xcolor}
\usepackage[textsize=small]{todonotes}
\usepackage{caption,subcaption}
\usepackage{soul}

\DeclareMathOperator*{\Tr}{Tr}

\usepackage[numbers]{natbib}

\begin{document}

	\newtheorem{proposition}{Proposition}[section]
	\newtheorem{theorem}[proposition]{Theorem}
	\newtheorem{corollary}[proposition]{Corollary}
	\newtheorem{lemma}[proposition]{Lemma}
	\newtheorem{conjecture}[proposition]{Conjecture}
	\newtheorem{question}[proposition]{Question}
	\newtheorem{definition}[proposition]{Definition}
	\newtheorem{comment}[proposition]{Comment}
	\newtheorem{algorithm}[proposition]{Algorithm}
	\newtheorem{assumption}[proposition]{Assumption}
	\newtheorem{condition}[proposition]{Condition}
	\numberwithin{equation}{section}
	\numberwithin{proposition}{section}

\newcommand{\skp}{\vspace{\baselineskip}}
\newcommand{\noi}{\noindent}
\newcommand{\osc}{\mbox{osc}}
\newcommand{\lfl}{\lfloor}
\newcommand{\rfl}{\rfloor}

\theoremstyle{remark}
\newtheorem{example}{\bf Example}[section]
\newtheorem{remark}{\bf Remark}[section]

\newcommand{\img}{\imath}
\newcommand{\iy}{\infty}
\newcommand{\eps}{\varepsilon}
\newcommand{\del}{\delta}
\newcommand{\Rk}{\mathbb{R}^k}
\newcommand{\RR}{\mathbb{R}}
\newcommand{\spa}{\vspace{.2in}}
\newcommand{\V}{\mathcal{V}}
\newcommand{\E}{\mathbb{E}}
\newcommand{\I}{\mathbb{I}}
\newcommand{\PP}{\mathbb{P}}
\newcommand{\sgn}{\mbox{sgn}}
\newcommand{\ti}{\tilde}

\newcommand{\QQ}{\mathbb{Q}}

\newcommand{\XX}{\mathbb{X}}
\newcommand{\XXz}{\mathbb{X}^0}

\newcommand{\lan}{\langle}
\newcommand{\ran}{\rangle}
\newcommand{\lf}{\lfloor}
\newcommand{\rf}{\rfloor}
\def\wh{\widehat}
\newcommand{\defn}{\stackrel{def}{=}}
\newcommand{\txb}{\tau^{\epsilon,x}_{B^c}}
\newcommand{\tyb}{\tau^{\epsilon,y}_{B^c}}
\newcommand{\tilxb}{\tilde{\tau}^\eps_1}
\newcommand{\pxeps}{\mathbb{P}_x^{\eps}}
\newcommand{\non}{\nonumber}
\newcommand{\dist}{\mbox{dist}}

\newcommand{\Om}{\mathnormal{\Omega}}
\newcommand{\om}{\omega}
\newcommand{\vph}{\varphi}
\newcommand{\Del}{\mathnormal{\Delta}}
\newcommand{\Gam}{\mathnormal{\Gamma}}
\newcommand{\Sig}{\mathnormal{\Sigma}}

\newcommand{\tilyb}{\tilde{\tau}^\eps_2}
\newcommand{\beq}{\begin{eqnarray*}}
\newcommand{\eeq}{\end{eqnarray*}}
\newcommand{\beqn}{\begin{eqnarray}}
\newcommand{\eeqn}{\end{eqnarray}}
\newcommand{\ink}{\rule{.5\baselineskip}{.55\baselineskip}}

\newcommand{\bt}{\begin{theorem}}
\newcommand{\et}{\end{theorem}}
\newcommand{\deps}{\Del_{\eps}}
\newcommand{\dbl}{\mathbf{d}_{\tiny{\mbox{BL}}}}

\newcommand{\be}{\begin{equation}}
\newcommand{\ee}{\end{equation}}
\newcommand{\ac}{\mbox{AC}}
\newcommand{\BB}{\mathbb{B}}
\newcommand{\VV}{\mathbb{V}}
\newcommand{\DD}{\mathbb{D}}
\newcommand{\KK}{\mathbb{K}}
\newcommand{\HH}{\mathbb{H}}
\newcommand{\TT}{\mathbb{T}}
\newcommand{\CC}{\mathbb{C}}
\newcommand{\ZZ}{\mathbb{Z}}
\newcommand{\SSS}{\mathbb{S}}
\newcommand{\EE}{\mathbb{E}}
\newcommand{\NN}{\mathbb{N}}
\newcommand{\MM}{\mathbb{M}}

\newcommand{\clg}{\mathcal{G}}
\newcommand{\clb}{\mathcal{B}}
\newcommand{\cls}{\mathcal{S}}
\newcommand{\clc}{\mathcal{C}}
\newcommand{\clj}{\mathcal{J}}
\newcommand{\clm}{\mathcal{M}}
\newcommand{\clx}{\mathcal{X}}
\newcommand{\cld}{\mathcal{D}}
\newcommand{\cle}{\mathcal{E}}
\newcommand{\clv}{\mathcal{V}}
\newcommand{\clu}{\mathcal{U}}
\newcommand{\clr}{\mathcal{R}}
\newcommand{\clt}{\mathcal{T}}
\newcommand{\cll}{\mathcal{L}}
\newcommand{\clz}{\mathcal{Z}}
\newcommand{\clq}{\mathcal{Q}}
\newcommand{\clo}{\mathcal{O}}

\newcommand{\cli}{\mathcal{I}}
\newcommand{\clp}{\mathcal{P}}
\newcommand{\cla}{\mathcal{A}}
\newcommand{\clf}{\mathcal{F}}
\newcommand{\clh}{\mathcal{H}}
\newcommand{\N}{\mathbb{N}}
\newcommand{\Q}{\mathbb{Q}}
\newcommand{\bfx}{{\boldsymbol{x}}}
\newcommand{\bfa}{{\boldsymbol{a}}}
\newcommand{\bfh}{{\boldsymbol{h}}}
\newcommand{\bfs}{{\boldsymbol{s}}}
\newcommand{\bfm}{{\boldsymbol{m}}}
\newcommand{\bff}{{\boldsymbol{f}}}
\newcommand{\bfb}{{\boldsymbol{b}}}
\newcommand{\bfw}{{\boldsymbol{w}}}
\newcommand{\bfz}{{\boldsymbol{z}}}
\newcommand{\bfu}{{\boldsymbol{u}}}
\newcommand{\bfell}{{\boldsymbol{\ell}}}
\newcommand{\bfn}{{\boldsymbol{n}}}
\newcommand{\bfd}{{\boldsymbol{d}}}
\newcommand{\bfbeta}{{\boldsymbol{\beta}}}
\newcommand{\bfzeta}{{\boldsymbol{\zeta}}}
\newcommand{\bfnu}{{\boldsymbol{\nu}}}
\newcommand{\bfvarphi}{{\boldsymbol{\varphi}}}

\newcommand{\curvz}{{\bf \mathpzc{z}}}
\newcommand{\curvx}{{\bf \mathpzc{x}}}
\newcommand{\curvi}{{\bf \mathpzc{i}}}
\newcommand{\curvs}{{\bf \mathpzc{s}}}
\newcommand{\blip}{\mathbb{B}_1}
\newcommand{\loc}{\text{loc}}

\newcommand{\BM}{\mbox{BM}}

\newcommand{\tac}{\mbox{\scriptsize{AC}}}



\begin{frontmatter}
\title{Diffusion Approximations for Controlled Weakly Interacting Large Finite State Systems with Simultaneous Jumps	
}

 \runtitle{Diffusion Approximations for Controlled Weakly Interacting Systems}

\begin{aug}
\author{Amarjit Budhiraja  and Eric Friedlander\\ \ \\
}
\end{aug}

\today

\skp

\begin{abstract}
We consider a rate control problem  for an $N$-particle weakly interacting finite state Markov process. 
The process models the state evolution of a large collection of particles and allows for multiple particles to change state simultaneously.
Such models have been proposed for large communication systems (e.g. ad hoc wireless networks) but are also suitable for other settings such as chemical-reaction networks. An associated diffusion control problem is presented and we show that the value function of the $N$-particle controlled system converges to the value function of the limit diffusion control problem as $N\to\infty$. The diffusion coefficient in the limit model is typically degenerate, however under suitable conditions there is an equivalent formulation in terms of a controlled diffusion with a uniformly non-degenerate diffusion coefficient. Using this equivalence, we show that near optimal continuous feedback controls exist for the diffusion control problem. We then construct near asymptotically optimal control policies for the $N$-particle system based on such continuous feedback controls. Results from some numerical experiments are presented.

\noi {\bf AMS 2000 subject classifications:} Primary 60K35, 60H30, 93E20; secondary 60J28, 60J70, 60K25, 91B70.

\noi {\bf Keywords:} Mean field approximations, diffusion approximations, asymptotic optimality, stochastic networks, stochastic control, propagation of chaos, rate control, ad hoc wireless networks.
\end{abstract}

\end{frontmatter}

\section{Introduction}\label{introsec}

We study a pure jump, weakly interacting, Markovian particle system in which jump rates can be dynamically modulated by a controller. The stochastic system of interest describes the state evolution of a collection of $N$ particles where each particle's state takes values in a finite set $\XX$. By a weak interaction we mean that the jump rates for a typical particle depend on the states of the remaining particles through the empirical distribution of particle states. System dynamics will allow for multiple particles to change states simultaneously, but there will be a fixed finite number of jump types. Such jump-Markov processes have been proposed as models for ad hoc wireless networks \cite{antunes2008stochastic} of the following form. Consider a system of $N$ finite capacity servers (particles/nodes). Jobs of $K$ different types, each with their own capacity requirement, arrive at each node at rate $\lambda_k, k=1,\ldots,K$ and are admitted if there is enough available capacity. All the jobs in the system of type $k$ have exponential residence time with mean $\tau^{-1}_k$. After an exponential holding time with mean $\gamma^{-1}_k$ a job of type $k$ will attempt to switch to another server which is chosen uniformly at random, and is admitted if there is available capacity, otherwise the job is lost. The state of a particle describes the number of various types of jobs being processed at the server. Under conditions, by classical results, the stochastic process of particle state empirical measures converges to the solution of a $d$-dimensional ordinary differential equation (ODE) (cf. \cite{kurtz1970solutions}), where $d=|\XX|$. This ODE captures the nominal behavior of the system over time as $N$ becomes large.

Taking a different perspective, the analysis of such ODE is a natural starting point for system design. By studying the mapping between system parameters and solution sets of the ODE one can identify parameter values that lead to desirable system behavior over time, at least in the law of large number limit as determined by the solution of the ODE. However, even when the system has been designed to reproduce a certain targeted nominal behavior the actual stochastic process of interacting particles may deviate significantly from the behavior determined by the ODE. It then becomes of interest to study dynamic control algorithms that modulate controllable system parameters to nudge the stochastic process closer to its desired nominal behavior. In general, adjusting system parameters incurs a cost and thus there is a trade off between this and the cost for deviating from the nominal behavior. A natural approach for analyzing this trade off is through an optimal stochastic control formulation where the controller seeks to minimize a suitable cost function which accounts for both types of costs noted above. 

The goal of this work is to develop a systematic stochastic control framework for studying optimal regulation of large, weakly interacting, pure jump Markov processes that arise from problems in communication networks. Since the jump rates in the system are of $\clo(N)$, and in a typical system $N$ is large, an exact analysis of this control problem becomes computationally intractable and thus one seeks a suitable approximate approach. The basic idea is to consider a sequence of networks indexed by $N$ such that the given physical system is embedded in this sequence for some fixed large value of $N$. A suitable asymptotic model, as $N\to\iy$, is used as a surrogate for the control problem in the $N$-th network. The asymptotic model taken here is based on diffusion approximations which give the limit behavior of fluctuations of the empirical measure process from its LLN limit. In an uncontrolled setting, such diffusion limits can be derived from classical martingale problem techniques \cite{kurtz1971limit,joffe1986weak} that are also the starting point here for developing an asymptotic framework for the study of the optimal stochastic control problem. Diffusion approximation methods have been used extensively in stochastic network theory, in particular they have been very useful in the study of critically loaded stochastic processing networks (see \cite{kushner2013heavy,harrison1988brownian,atar2014asymptotic, bellwill01, DaiLin, wwbook, AA2, BGL} and references therein). In this context, diffusion processes arise as approximations for a fixed number of centered renewal processes with rates approaching infinity. Limit theorems and the scaling regime considered in these works (number of nodes is fixed, traffic intensity approaches 1) is quite different from the one where the number of nodes (particles) approaches infinity that is considered here. In communication systems that motivate study of such interacting processes, jumps correspond to either an admission of a job to one of the $N$ nodes in the system, transfer of a job from one node to another node, or the completion/rejection of a job (and thus exit from the system). We consider a formulation in which controls can make ``small'' adjustments to the rate values in order to nudge the system toward its nominal state. Specifically, the overall rate of jumps in the system is $\clo(N)$ whereas the allowable rate controls will be $\clo(\sqrt{N})$. Although the magnitude of control becomes negligible compared to the overall rate as $N$ becomes large, in the diffusion scaling such a control can lead to an appreciable improvement in performance (see Section \ref{Ex1} for some numerical results). In the law of large numbers limit the controlled and uncontrolled systems both converge to the same nominal behavior as expected, but the diffusion limit of the two systems will in general differ in the drift coefficient. In particular, under suitable feedback controls the centered and normalized controlled process will converge to a diffusion with a nonlinear (in state) drift term whereas the uncontrolled process will converge to a time inhomogeneous Gauss-Markov process. In terms of cost, one can consider various types of criteria, but for simplicity we restrict ourselves to a finite time horizon cost where the running cost is a sum of two terms. The first term is  a continuous function, with at most polynomial growth, of the state of the centered and normalized empirical measure, and the second is a finite convex function of the (normalized) control.

Rather than attempting to look for an optimal control for the stochastic control problem for a fixed value of $N$, i.e. for the $N$-th system, we instead focus on  the more tractable goal of asymptotic optimality. More precisely, we are interested in constructing a sequence of control policies (indexed by $N$) such that the cost associated with the $N$-th system under the $N$-th control policy converges to the smallest possible value as $N\to\iy$. Analogous notions of asymptotic optimality are routinely used in heavy traffic analysis of queuing networks \cite{kushner2013heavy, harrison1988brownian,atar2014asymptotic, bellwill01, DaiLin, AA2,BGL}, but in the current work they are introduced in a very different asymptotic regime.  The key ingredient in the approach is to formulate and analyze a closely related stochastic control problem for diffusion processes. Roughly speaking, the state process in the diffusion control problem is the asymptotic analogue of the centered and normalized empirical measure process as $N\to\iy$. The control enters in the drift of the diffusion process whereas the diffusion coefficient is a non-random function of time. Our main result, Theorem \ref{thm:main}, shows that the diffusion control problem is a good approximation of the control problem for the $N$-th system, when $N$ is sufficiently large. Specifically, this theorem says that the value function associated with the control problem for the $N$-th system converges to the value function of the limit diffusion control problem. The key ingredients in the proof are Theorems \ref{thm:6}, \ref{thm:3}, and \ref{thm:5}. Theorem \ref{thm:6} gives the lower bound, namely it shows that the value function of the $N$-th system, asymptotically as $N\to\iy$, is bounded below by the value function of the diffusion control problem. The key steps in the proof are to establish suitable tightness properties of the sequence of scaled state and control processes and the characterization of the weak limit points. For the first step it is convenient to work with the relaxed control formulation (cf. \cite{kushner2013heavy, borkar1989optimal}) through which one can view controls as elements of a tractable Polish space. The second step proceeds via classical martingale problem techniques (cf. \cite{stroock2007multidimensional, ethier2009markov, joffe1986weak}).
%
%
%
Theorems \ref{thm:3} and \ref{thm:5} give the main steps needed for the complementary upper bound. For this bound, the main idea is to show that for any fixed $\eps>0$, there exists an $\eps$-optimal \textit{continuous} feedback control for the diffusion control problem (Theorem \ref{thm:5}), and that any such feedback control can be used to construct a sequence of control policies for the interacting particle system such that the associated costs converge to the cost under the feedback policy for the diffusion control problem (Theorem \ref{thm:3}). We begin, in Theorem \ref{thm:4}, by arguing  that for the diffusion control problem the infimum over all admissible controls is the same
as that over the class of feedback controls. Proof of this proceeds via certain conditioning arguments and PDE characterization results (cf. \cite{borkar1989optimal}) that allow the construction of a feedback control associated with any given admissible control such that the cost corresponding to the feedback control is no larger than that of the given admissible control. The result says that one can find an $\varepsilon$-optimal
control in the space of feedback controls. Although any such control corresponds to a natural collection of control policies for the sequence of $N$-particle systems, in order to prove the convergence of associated costs, which once more is based on martingale problem methods, we require additional regularity properties of the feedback control.  The key step is Theorem \ref{thm:5}
%
%
%
%
that shows that for any feedback control $g$ there exists a sequence of continuous feedback controls $\{g_n\}$ for the limit diffusion control problem such that the associated sequence of controlled diffusions converge weakly to the diffusion under the feedback control $g$. The proof requires some estimates based on an application of Girsanov's theorem which, in turn, relies on the non-degeneracy of the diffusion coefficient. Although the controlled diffusion that describes the asymptotic model is degenerate, we show that there is an equivalent formulation in terms of a ($d-1$)-dimensional controlled diffusion which is uniformly non-degenerate under suitable assumptions. This equivalent representation, in addition to providing a  feedback control of the desired form, is also key in proving weak uniqueness for stochastic differential equations (SDE) describing limit state processes associated with feedback controls.

In Section \ref{Ex1}, we will illustrate our approach through a numerical example. This example is 
the controlled analogue of a  model introduced in \cite{antunes2008stochastic}, and one can approach more general forms of this model along similar lines. The running cost function we consider is quadratic in the normalized state and control processes. The corresponding limit diffusion control problem in this case becomes the classical stochastic linear quadratic regulator (LQR) with time dependent coefficients (see \cite{fleming1976deterministic}). The optimal feedback control for the diffusion control problem can be given explicitly by solving a suitable Riccati equation. Our numerical results show that implementation of the control policy based on the optimal feedback control for the limit LQR to a system with $N=10,000$ leads to an improvement of up to 15.5\% on the cost for the uncontrolled system. A more detailed numerical analysis of the implementation of such diffusion approximation based control schemes will be presented elsewhere.

The paper is organized as follows. Section \ref{Main} presents the precise system of weakly interacting pure jump processes considered here. We will also present key assumptions and the main result of this work. Sections \ref{WeakJump} and \ref{ContSys} describe the uncontrolled and controlled systems, respectively. Assumptions which ensure convergence of the system to its fluid limit are introduced for both cases. Section \ref{ContSys} also introduces the cost criteria that is considered in this work. Section \ref{DiffCont} presents the diffusion control problem that formally corresponds to the limit as $N\to\iy$ of the control problem for the $N$-th system. The section also introduces the key non-degeneracy assumption (Condition \ref{con:8}) that is needed in order to obtain weak uniqueness of SDE with feedback controls and existence of near optimal continuous feedback controls. We also introduce our main assumptions on the controlled rate functions (Conditions \ref{con:6} and \ref{con:7}). In Section \ref{MainRes} we present our main result, namely Theorem \ref{thm:main}. Section \ref{Ex1} presents results from a numerical study. The remainder of this work is devoted to proof of Theorem \ref{thm:main}. In Section \ref{Tight} we present a key tightness result which is used both in the proof of the upper and lower bound. In Section \ref{LBnd} (see Theorem \ref{thm:6}) we prove the lower bound that was discussed earlier in the Introduction. In preparation for the proof of the upper bound, we introduce the class of feedback controls in Section \ref{Feedback}. Sections \ref{NetFeed} and \ref{DiffFeed} describe such controls for the prelimit system and the limit diffusion model, respectively. Section \ref{FeedConv} constructs a sequence of prelimit control policies from an arbitrary continuous feedback control for the diffusion control problem such that the cost for the particle systems under the sequence of control policies converges to the cost of the corresponding controlled diffusion. Finally in Section \ref{NearOptConCont}, we show that the infimum of the cost for the limit diffusion over all admissible controls is the same as that over the class of feedback controls and that there exist continuous feedback controls which are $\eps$-optimal. The results from sections \ref{LBnd}, \ref{Feedback}, and \ref{NearOptConCont}, (namely Theorems \ref{thm:6}, \ref{thm:3}, and \ref{thm:5}) together give our main result, Theorem \ref{thm:main}.

\subsection{Notation}\label{Notation}
The following notation will be used. 
We will use the notations $\{X_t\}$ and $\{X(t)\}$ interchangeably for stochastic processes.
The space of probability measures on a Polish space $\SSS$, equipped with the topology of weak convergence, will be denoted by $\clp(\SSS)$.  For $\SSS$ valued random variables $X$, $X_N$, $N\ge 1$, convergence in distribution of $X_N$ to $X$ as $N\to \infty$ will be denoted as $X_N \Rightarrow X$.
The Borel $\sigma$-field on a Polish space $\SSS$ will be denoted as $\clb(\SSS)$. 
The space of functions that are right continuous with left limits (RCLL) from $[0,T]$  to $\SSS$
will be denoted as $\DD([0,T]:\SSS)$  and equipped with the usual Skorohod topology.
Similarly $\CC([0,T]:\SSS)$  will be the space of continuous functions
from $[0,T]$  to $\SSS$,  equipped with the  uniform topology.

We will usually denote by $\kappa, \kappa_1, \kappa_2, \cdots$, the constants that appear in various estimates within a proof. The values of these constants may change from one proof to another. 
Cardinality of a finite set $A$ will be denoted as $|A|$. We will denote by $B(r)$  the $L^1$ ball of radius $r$ centered at the origin in some Euclidean space $\RR^d$. 
The Euclidean norm of a $d$-dimensional vector or a $d\times d$ matrix will be denoted as $\|\cdot\|$.
The linear span of a set $A \subset \RR^d$ will be denoted as $\text{Sp} A$.
The space of continuous (resp. continuous and bounded) functions from metric space $\SSS_1$ to $\SSS_2$ will be denoted as $\CC(\SSS_1:\SSS_2)$ (resp. $\CC_b(\SSS_1:\SSS_2)$). 
When $\SSS_2 = \RR$ we sometimes abbreviate this notation and write $\CC(\SSS_1)$ and $\CC_b(\SSS_1)$.
For a bounded function $f: \SSS \to \RR$,  $\|f\|_{\infty}\doteq \sup_{x \in \SSS} |f(x)|$.
The space of real valued continuous functions defined on $\RR^d$ whose first $k\in\NN$ (resp. all) derivatives exist and are continuous will be denoted $\CC^k(\RR^d)$ (resp. $\CC^{\iy}(\RR^d)$).  
We denote the subset of $\CC^k(\RR^d)$ of functions with compact support as $\CC_c^k(\RR^d)$.
Similarly $\CC^{1,2}([0,T]\times\RR^d)$ denotes the space of functions from $(0,T)\times\RR^d$ to $\RR$ that are once continuously differentiable  in the time coordinate, twice continuously differentiable in the space coordinate, and are such that the function and its derivatives can be continuously extended to $[0,T]\times\RR^d$. The space of $m\times n$ dimensional matrices whose entries take values in a set $\SSS$ will be denoted $\MM^{m\times n}(\SSS)$. For $M\in\MM^{m\times n}(\SSS)$, $M_{i,j}$ will the denote that entry of $M$ which is in the $i$-th row and $j$-th column. 
The transpose of a matrix $M$ will be denoted as $M'$ and trace of a square matrix $M$ will be denoted as $\Tr(M)$.
$\mathbf{1}$ and $I$ will denote the matrix of 1's and the identity matrix, respectively, the dimension of which will be context dependent.
For a Polish space $\SSS$ we denote by $\clm(\SSS)$ the space of all locally finite measures on $\SSS$. This space will be equipped with the usual vague topology, namely, the weakest topology such that for every $f\in\CC_b(\SSS)$ with compact support,
\begin{align*}
\nu\mapsto\int_\SSS f(u)\nu(du),\ \nu\in\clm(\SSS),
\end{align*}
is continuous.

\section{Problem Formulation and Main Results}\label{Main}
In this section we will describe the basic control problem of interest and give a precise mathematical formulation. We begin by introducing the uncontrolled pure jump Markov process in Section \ref{WeakJump} and recall a classical law of large numbers result for such systems. Section \ref{ContSys} will present the controlled system that we study and also our cost criteria. In Section \ref{DiffCont} we will introduce our main assumptions on the controlled rate matrices and based on these assumptions introduce a control problem for diffusion processes that can formally be regarded as the limit of control problems considered in Section \ref{ContSys}. Finally, in Section \ref{MainRes} we present our main result. This result says in particular that  a suitable near optimal diffusion control can be used to construct a sequence of control policies for the particle system in Section  \ref{ContSys} that are asymptotically near optimal. For a numerical example that illustrates the application of the result, we refer the reader to Section \ref{Ex1} where we present a model from communication networks that is a controlled version of some models introduced in \cite{antunes2008stochastic} and which falls within the framework considered here.

\subsection{Weakly Interacting Jump Markov Process}\label{WeakJump}
Fix $T \in (0, \infty)$.  All stochastic processes in this work will be considered on the time horizon $[0,T]$.
Consider a system of $N$ particles where the state of each particle takes values in the set $\XX=\{1,\ldots,d\}$. The evolution of the system is described by an $N$-dimensional pure jump Markov process $\mathbf{X}_N(t)=\{X_N^1(t),\ldots,X_N^N(t)\}$ where $X_N^i(t)$ represents the state of particle $i$ at time $t$. The system allows multiple particles to change state at a given time, but restricts such jumps to $K$ transition types; in particular the $k$-th transition type can only affect at most $n_k$ particles, $k\in\textbf{K}\doteq\{1,\ldots,K\}$. The jump intensity is state dependent, however the state dependence is of the following specific form: Denoting for $x\in\XX^N$, the probability measure $\{\frac{1}{N}\sum_{i=1}^NI_{\{x_i\}}(m)\}_{m\in\XX}$ on $\XX$ by $\{\zeta_N^m(x)\}_{m\in\XX}$, the jump intensity at the instant $t$ is a function of $\zeta_N(X_N(t))$. The set of jumps and the corresponding transition rates can be described in terms of the subset $\MM_N$ of $\MM^{d\times d}(\NN_0)$ consisting of all matrices with zeroes on the diagonal and  with sum of all entries at most  $N$, as follows. To any $k\in\mathbf{K}$ we associate a map $\Psi_N^k:\clp(\XX)\times\MM_N\to\RR_+$ such that for $x\in\XX^N,\ \Psi_N^k(\zeta_N(x),\Theta)=0$ if
\begin{align}\label{eqn:rateineq}
\sum_{i,j}\Theta_{i,j}>n_k \text{ or }\sum_{j=1}^d\Theta_{i,j}>N\zeta^i_N(x),\ i=1,\ldots,d.
\end{align}
Roughly speaking, $\Psi_N^k(\zeta_N(x),\Theta)$ will give the rate  of type $k$ jumps (associated with $\Theta$) when the system is in state $x\in\XX^N$. A type $k$ jump associated with $\Theta\in\MM_N$ corresponds to $\Theta_{ij}$ particles simultaneously jumping from state $i$ to state $j$, for all $i\neq j$ and $i,j=1,\ldots,d$. Thus the first inequality in \eqref{eqn:rateineq} says that at most $n_k$ particles change states under a jump of type $k$, while the second inequality says that a jump of type $k$ can occur only when there are enough particles to participate in it. In terms of $\Psi_N^k$ the overall rate of jumps of type $k$ associated with $\Theta$, when the system is in state $x\in\XX^N$, is given as
\begin{align*}
\Psi_N^k(\zeta_N(x),\Theta)\prod_{m=1}^d\binom{N\zeta^m_N(x)}{\sum_{j=1}^d\Theta_{m,j}}
\binom{\sum_{j=1}^d\Theta_{m,j}}{\Theta_{m,1},\ldots,\Theta_{m,d}}
\end{align*}
and such a jump takes a state $x\in\XX^N$ to a state $\ti x\in\XX^N$ where 
$$N\zeta_N^m(\ti x)=N\zeta_N^m(x)+\sum_{i=1}^d\Theta_{i,m}-\sum_{j=1}^d\Theta_{m,j},\; m=1,\ldots,d. $$

A more convenient description of this system is given through the pure jump Markov process $\{\mu_N(t)\}$ where $\mu_N(t)\doteq\zeta_N(X_N(t))$ represents the empirical measure of the particle states. We will identify the space of probability measures, $\clp(\XX)$, with the $d$-dimensional simplex, $\cls\doteq\{(x_1,\ldots,x_d)\in\RR_+^d|\sum_{i=1}^dx_i=1\}$. Similarly, we will identify $\clp_N(\XX)$, the space of all $\mu \in \clp(\XX)$ such that
$\mu\{j\} \in \frac{1}{N}\NN$ for all $j \in \XX$, with $\cls_N=\cls\cap\frac{1}{N}\NN^d$. Let, for $k\in\mathbf{K}$,
\begin{align*}
\Del^k
&\doteq \left\{(I,J)\in\NN_0^d\times\NN_0^d:\quad\sum_{x\in\XX}I_x=\sum_{x\in\XX}J_x\leq n_k,\quad\sum_{x\in\XX}|J_x-I_x|>0\right\},
\end{align*}
and for $\nu=(I,J)\in\Del^k$ let
\begin{align*}
\Phi(\nu)
= \Phi(I,J)
&\doteq\left\{\Theta\in\MM_N\Big|\sum_{j=1}^d\Theta_{i,j}=I_i,\sum_{i=1}^d\Theta_{i,j}=J_j,\; i,j=1, \cdots , d\right\}.
\end{align*}
The jumps of $\{\mu_N(t)\}$ are described as follows. For each $k\in\mathbf{K}$ and $\nu=(I,J)\in\Del^k$ the empirical measure jumps from $r\mapsto r+\frac{1}{N}e_\nu$ with rate
\begin{align*} 
\bar{\Gam}_N^k(r,\nu)\doteq\sum_{\Theta\in\Phi(\nu)}\Psi_N^k(r,\Theta)\prod_{m=1}^d\binom{Nr^m}{\sum_{j=1}^d\Theta_{m,j}}
\binom{\sum_{j=1}^d\Theta_{m,j}}{\Theta_{m,1},\ldots,\Theta_{m,d}}
\end{align*}
where $r = (r^m)_{m=1}^d \in \cls_N$, $e_\nu\doteq\sum_{x\in\XX}(J_x-I_x)e_x$ and $e_x$ is the unit vector in $\RR^d$ with 1 at the $x$-th coordinate and 0 everywhere else. Thus a jump associated with $k\in\mathbf{K}$ and $\nu\in\Del^k$ corresponds to $I_x$ particles in state $x,\ x\in\XX$, simultaneously jumping to new states such that $J_y$ of the particles end up in state $y,\ y\in\XX$. A succinct description of the evolution of the Markov process $\mu_N(t)$ is through its infinitesimal generator which is given as
\begin{align}\label{eqn:gen}
\bar{L}^Nf(r)=\sum_{k\in\mathbf{K}}\sum_{\nu\in\Del^k}\bar{\Gam}_N^{k}(r,\nu)\left[f\left(r+\frac{1}{N}e_\nu\right)-f(r)\right],\qquad r\in\cls_N.
\end{align}
We will make the following assumption on the asymptotic behavior of the rates.
\begin{condition}\label{con:1} For all $k\in\mathbf{K}$ and $\nu\in\Del^k$ there exists a Lipschitz  function $r\mapsto\Gam^k(r,\nu)$ on $\cls$ such that 
\begin{align}\label{eqn:ratecon}
\limsup_{N\to\iy}\sup_{r\in\cls_N}\left|\frac{1}{N}\bar{\Gam}^k_N(r,\nu)-\Gam^k(r,\nu)\right|=0
\end{align}
\end{condition}
We now present a classical law of large numbers result that characterizes the limit, $\mu(t)$, of the pure jump Markov process $\mu_N(t)$ as $N\to\iy$. For a proof we refer the reader to Theorem 2.11 of \cite{kurtz1970solutions}.
\begin{proposition}\label{prop:2}
Define,
\begin{align}\label{eqn:8}
F(r)\doteq\sum_{k\in\mathbf{K}}\sum_{\nu\in\Del^k}\Gam^k(r,\nu)e_\nu,\ r\in\cls.
\end{align}
Suppose that $\mu_N(0)\to\mu_0$ in probability and Condition \ref{con:1} holds, then $\mu_N(t)\to \mu(t)$ uniformly on $[0,T]$, in probability, where $\mu(t)$ is the unique solution of the ODE 
\begin{align}\label{eqn:ODE}
\dot{\mu}(t)=F(\mu(t)),\ \mu(0)=\mu_0.
\end{align}
\end{proposition}

\subsection{Controlled System}\label{ContSys}
In this work we will study a controlled version of the Markov process introduced in Section \ref{WeakJump}. Roughly speaking, control action will allow perturbations of the rate function $\bar{\Gamma}_N^k$ that are of $\clo\left(\frac{1}{\sqrt{N}}\right)$. The goal of the controller is to minimize a suitable finite time horizon cost. A precise mathematical formulation is as follows. Let 
\begin{align}\label{eqn:15}
\mathbf{\ell}\doteq\sum_{k\in \mathbf{K}}|\Del^k|,
\end{align}
$\Lambda$ be a compact convex subset of $\RR^\mathbf{\ell}$, and $\Lambda_N=\frac{1}{\sqrt{N}}\Lambda$ for $N\in\NN$. $\Lambda_N$ will be the control set in the $N$-th system. Let $\{\Gam_N^k(r,u,\nu):r\in\cls_N,u\in\Lambda_N,k\in \mathbf{K},\nu\in\Del^k\}$ be a collection of non-negative real numbers. More precisely, $(r,u)\mapsto \Gam_N^k(r,u,\nu)$ is a map from $\cls_N\times \Lambda_N$ to $\RR_+$ for each $N\in\NN,\ k\in\mathbf{K},\ \nu\in\Del^k$. These correspond to the controlled rates in the $N$-th system. We now introduce the controlled stochastic processes associated with such controlled rates. 

Fix $N\in\NN$ and let $(\Om^N,\clf^N,\PP^N)$ be a probability space on which are defined unit rate mutually independent Poisson processes $\{\mathcal{N}_{k,\nu},\ k\in\mathbf{K},\ \nu\in\Del^k\}$. The processes $\{\mathcal{N}_{k,\nu}\}$ will be used to describe the stream of jumps corresponding to $k\in\mathbf{K},\ \nu\in\Del^k$. Let $U^N$ be a $\Lambda_N$-valued measurable process representing the rate control in the system. Under control $U^N$ the state process $\mu_N(\cdot)$ is given by the following equation:
\begin{align}\label{eqn:stpro}
\mu_N(t)
&= \mu_N(0)+\frac{1}{N}\sum_{k\in\mathbf{K}}\sum_{\nu\in\Del^k}e_\nu \mathcal{N}_{k,\nu}\left(\int_0^t\Gam_N^k(\mu_N(s),U^N(s),\nu)ds\right).
\end{align}
In order for such a control to be admissible it should satisfy suitable non-anticipative properties.
More precisely, $U^N$ is said to be an admissible control if, with some filtration $\{\clf^N_t\}$ on $(\Om^N,\clf^N,\PP^N)$, $U^N$ is $\{\clf^N_t\}$-progressively measurable, $\mu_N$ is $\{\clf^N_t\}$-adapted, and $\{M^N_{k,\nu},k\in\mathbf{K},\nu\in\Del^k\}$ defined below are $\{\clf^N_t\}$-martingales
\begin{align}\label{eqn:mart}
M_{k,\nu}^N(t)\doteq \frac{1}{N}\left(\mathcal{N}_{k,\nu}\left(\int_0^t\Gam_N^k(\mu_N(s),U^N(s),\nu)ds\right)-\int_0^t\Gam_N^k(\mu_N(s),U^N(s),\nu)ds\right)
\end{align}
with quadratic variation processes $\lan M^N_{k,\nu},M^N_{k',\nu'}\ran_t=\del_{(k,\nu),(k',\nu')}\frac{1}{N^2}\int_0^t\Gam_N^k(\mu_N(s),U^N(s),\nu)ds$ where $\del_{\alpha,\alpha'}$ equals 1 if $\alpha=\alpha'$ and 0 otherwise. We note that in general such a filtration will depend on the control. We denote the set of all such admissible controls as $\cla_N$. 

For a $U^N\in\cla_N$, define the process 
\begin{align}
V_N(s)=\sqrt{N}(\mu_N(s)-\mu(s))\label{eqn:veen}
\end{align}
where, as above, $\mu_N$ is the state process under control $U^N$. We consider a cost that is a function of the suitably normalized control action and the centered and normalized state of the system given through the process $\{V_N(\cdot)\}$. Specifically, we consider for $N\in\NN,\ x_N\in\cls_N$ a ``finite time horizon  cost'' associated with an admissible control $U^N\in\cla_N$ and initial condition $x_N$ as,
\begin{align}\label{eqn:costprelim}
J_N(U^N,v_N)\doteq\E\int_0^T   (k_1(V_N(s))+k_2(\sqrt{N}U^N(s)))ds
\end{align}
where $v_N=\sqrt{N}(x_N-\mu_0)$, $k_2\in \CC(\Lambda)$ is a nonnegative convex function, and $k_1\in\CC(\RR^d)$ is a nonnegative function with at most polynomial growth. I.e. there exists a $p>1$ and $C_{k_1}\in(0,\iy)$ such that $k_1(x)\leq C_{k_1}(1+\|x\|^p)$ for all $x\in\RR^d$. Define the corresponding value function to be
\begin{align*}
R_N(v_N)\doteq\inf_{U^N\in\cla_N}J_N(U^N,v_N).
\end{align*}

Computing an optimal control for the above problem for a given $N$ is, in general, challenging and computationally intensive. It is therefore of interest to consider approximate approaches. In the next section we introduce some conditions on the controlled rate matrices that will suggest a natural diffusion approximation for  this control problem.

\subsection{Diffusion Control Problem}\label{DiffCont}
We now introduce our main assumptions on the controlled rate matrices. The first two conditions make precise the requirement that controlled rates are $\clo\left(\frac{1}{\sqrt{N}}\right)$ perturbations of the nominal values given through $\{\Gamma^k,\ k\in\mathbf{K}\}$. In particular, the first condition will ensure that the controlled pure jump Markov process will converge to the same limit as the uncontrolled process $\mu_N$ in Section \ref{WeakJump} under the law of large number scaling.
\begin{condition}\label{con:3}With $\{\Gam^k(r,\nu),\ k\in\mathbf{K},\ \nu\in\Del^k,\ r\in\cls\}$ as in Condition \ref{con:1}
\begin{align}\label{eqn:contrateconv}
\limsup_{N\to\iy}\sup_{r\in\cls_N}\sup_{u\in \Lambda_N}\left|\frac{1}{N}\Gam^{k}_N(r,u,\nu)-\Gam^{k}(r,\nu)\right|=0.
\end{align}
\end{condition}

We next introduce a strengthening of Condition \ref{con:3} that will play a key role in the proof of tightness of the sequence $\{V_N\}$ of controlled state processes.
\begin{condition}\label{con:6} 
There exists a $C_1\in(0,\iy)$ such that for every $N\in\NN$
\begin{align}\label{eqn:tightcond}
\sup_{u\in \Lambda_N}\sup_{\xi\in\cls_N(y)}\sqrt{N}\left|\frac{1}{N}\Gam_N^{k}\left(\frac{1}{\sqrt{N}}y+\xi,u,\nu\right)-\Gam^{k}\left(\xi,\nu\right)\right|\leq C_1(1+\|y\|)
\end{align}
for all $k\in \mathbf{K},\ \nu\in\Del^k$, and $y\in B(2\sqrt{N})\subset\RR^d$ where $\cls_N(y)=\{\xi\in\cls:\frac{1}{\sqrt{N}}y+\xi\in\cls_N\}$.
\end{condition}
Taking $y=0$ in \eqref{eqn:tightcond} we see that Condition \ref{con:6} implies that there exists a $C_2\in(0,\iy)$ such that 
\begin{align}\label{eqn:1}
\sup_{N\geq 1}\sup_{r\in\cls_N}\sup_{u\in\Lambda_N}\frac{1}{N}\Gam^k_N(r,u,\nu)\leq C_2
\end{align}
for all $k\in\mathbf{K},\ \nu\in\Del^k$. Note also that Condition \ref{con:6} implies Condition \ref{con:3}.

The next condition will identify the drift term in our limit diffusion control problem.
Note that any $u\in\Lambda$ (or $\Lambda_N$) can be indexed by $k\in\mathbf{K}$ and $\nu\in\Del^k$ and we will denote the corresponding entry by $u_{k,\nu}$.
\begin{condition}\label{con:7} 
There exist, for each $k\in \mathbf{K}, \nu\in\Del^k$, bounded functions $h_1^k(\nu,\cdot):\cls\to\RR$ and $h_2^k(\nu,\cdot):\cls\to\RR^d$ such that for $u\in\Lambda$, $\xi\in\cls$, $y\in\RR^d$, with
\begin{align*}
H^k(y,\xi,u,\nu)\doteq h_1^k(\nu,\xi)u_{k,\nu}+h^k_2(\nu,\xi)\cdot y,
\end{align*}
we have for all compact $A\subset\RR^d$,
\begin{align}\label{eqn:betazero}
\limsup_{N\to\iy}\sup_{u\in\Lambda}\sup_{y\in A}\sup_{\xi\in\cls_N(y)}\left|\beta_k^N(y,\xi,u,\nu)\right|=0
\end{align}
where for $N\in\NN,k\in\mathbf{K}$, and $\nu\in\Del^k$, we define $\beta_k^N(\cdot,\cdot,\cdot,\nu):\RR^d\times\cls\times\Lambda\to\RR$ as
\begin{align*}
\beta_k^N(y,\xi,u,\nu)\doteq \sqrt{N}\left(\frac{1}{N}\Gam_N^k\left(\frac{1}{\sqrt{N}}y+\xi,\frac{1}{\sqrt{N}}u,\nu\right)-\Gam^k(\xi,\nu)\right)-H^k(y,\xi,u,\nu),
\end{align*}
if $\xi \in \cls_N(y)$ and $0$ otherwise.
\end{condition}

Define $\eta:[0,T]\times\RR^\mathbf{\ell}\to\RR^d$ and $\beta:[0,T]\to\RR^{d\times d}$ as
\begin{align}\label{eqn:16}
\eta(t,u)\doteq \sum_{k\in\mathbf{K}}\sum_{\nu\in\Del^k}\left(h_1^k(\nu,\mu(t))u_{k,\nu}\right)e_\nu \text{ and } \beta(t)\doteq\sum_{k\in\mathbf{K}}\sum_{\nu\in\Del^k}e_\nu[h_2^k(\nu,\mu(t))]'
\end{align}
Note that 
\begin{align}\label{eqn:Hdef}
\sum_{k\in\mathbf{K}}\sum_{\nu\in\Del^k}H^k(y,\mu(t),u,\nu)e_\nu=\eta(t,u)+\beta(t)y,\ t \in [0,T],y\in\RR^d.
\end{align}
Let $a:[0,T]\to\RR^{d\times d}$ be defined as
\begin{align*}
a(t) \doteq \sum_{k\in\mathbf{K}}\sum_{\nu\in\Del^k}(\Gam^k(\mu(t),\nu))e_\nu e_\nu'.
\end{align*}

The $d\times d$ matrix $a(t)$ will be the square of the diffusion coefficient for the limit controlled diffusion process. Note that $a(t)$ is a singular matrix since $e_\nu\cdot\mathbf{1}=0$ for all $k\in\mathbf{K}$ and $\nu\in\Del^k$. Let $Q=[q_1\ldots q_d],\ q_k\in\RR^d,$ be a $d\times d$ orthogonal matrix (i.e $QQ'=Q'Q=I$) such that $q_d=\frac{1}{\sqrt{d}}\mathbf{1}$. Then, in view of the above observation,
\begin{align}\label{eqn:alphadef}
Q'a(t)Q = \left(\begin{matrix}
\alpha(t) & 0\\
0 & 0
\end{matrix}\right)
\end{align}
where $\alpha(\cdot)$ is a Lipschitz, nonnegative definite, $(d-1)\times(d-1)$ matrix valued function. Let $\alpha^{1/2}(t)$ be the symmetric square root of $\alpha(t)$. Since $t\mapsto\alpha(t)$ is continuous so is $t\mapsto\alpha^{1/2}(t)$ (see e.g. \cite{chen1997continuity}). Define 
\begin{align}\label{eqn:7}
\sigma(t)\doteq Q
\left[
\begin{matrix}
\alpha^{1/2}(t) & 0\\
0 & 0\end{matrix}
\right]Q'.
\end{align}

The main goal of this paper is to show that  an optimal control problem for certain diffusion processes can be used to construct asymptotically near optimal control policies for the sequence of controlled systems in Section \ref{ContSys}. We now introduce this diffusion control problem. Let $(\Om,\clf,\PP,\{\clf_t\})$ be a filtered probability space with a $d$-dimensional $\{\clf_t\}$-Brownian motion $\{W_t\}$. We refer to $(\Om,\clf,\PP,\{\clf_t\},\{W_t\})$ as a system and denote it by $\Xi$. Denote the collection of $\clf_t$-progressively measurable, $\Lambda$ valued processes as $\cla(\Xi)$. This collection will represent the set of admissible controls for the diffusion control problem. The initial condition $v_0$ for our controlled diffusion process will lie in the set $\VV_{d-1}=\{x\in\RR^d|x\cdot\mathbf{1}=0\}$. For $U\in\cla(\Xi)$ and $v_0\in\VV_{d-1}$, let $V$ be the unique pathwise solution of
\begin{align}\label{eqn:6}
V(t)=v_0+\int_0^t\eta(s,U(s))ds+\int_0^t\beta(s)V(s)ds+\int_0^t\sigma(s)dW(s)
\end{align}
where $\eta,\beta$ are as introduced in \eqref{eqn:16} and $\sigma$ is as in \eqref{eqn:7}. Define the cost associated with $U\in\cla(\Xi)$ and $v_0\in\VV_{d-1}$ as
\begin{align}\label{eqn:costdiff}
J(U,v_0)\doteq\E\int_0^T   (k_1(V(s))+k_2(U(s)))ds.
\end{align}
The value function associated with the above diffusion control problem is
\begin{align*}
R(v_0)\doteq \inf_{\Xi}\inf_{U\in\cla(\Xi)}J(U,v_0),
\end{align*}
where the outside infimum is taken over all possible systems $\Xi$.

Although the matrix $\sigma(t)$ is singular for each $t$, the following condition will ensure that the dynamics of $V$ restricted to a certain $(d-1)$-dimensional subspace is non-degenerate.
\begin{condition}\label{con:8}
 There exists a $\Del^*\subset\cup_{k\in\mathbf{K}}\Del^k$ such that  $\text{Sp} \{e_\nu:\nu\in\Del^*\}$ equals $\VV_{d-1}$, and for every $\nu\in\Del^*$ there is a $k_\nu\in\mathbf{K}$ such that $\nu\in\Del^{k_\nu}$ and 
\begin{align*}
\kappa(T)\doteq\inf_{\nu\in\Del^*}\inf_{0\leq t\leq T}\Gam^{k_\nu}(\mu(t),\nu)>0.
\end{align*}
\end{condition}

The following lemma shows that under Condition \ref{con:8}, $\alpha$ is uniformly non-degenerate on compact sets. 
\begin{lemma}\label{lem:nondeg}
Under Condition \ref{con:8},  $\{\alpha(t):t\in[0,T]\}$ is a uniformly positive definite collection, namely, there exists a $C(T)\in(0,\iy)$ such that $x'\alpha(t)x\geq C(T)\|x\|^2$ for all $x\in\RR^{d-1}$ and $0\leq t\leq T$.
\end{lemma}
\begin{proof}
We first show that the matrix $G=\sum_{\nu\in\Del^*}e_\nu e_\nu'$ satisfies, for some $C_G\in(0,\iy)$, 
\begin{align}\label{eqn:2}
\xi'G\xi\geq C_G\|\xi\|^2
\end{align}
for all $\xi\in\VV_{d-1}$. For this it satisfies to check that for any nonzero $\xi\in\VV_{d-1}$, $\xi'G\xi>0$.

Suppose for some nonzero $\xi\in\VV_{d-1}$, $\xi'G\xi=0$. Since $\xi'G\xi=\sum_{\xi\in\Del^*}|\xi\cdot e_\nu|^2$ and $\text{Sp}\{e_\nu:\nu\in\Del^*\}=\VV_{d-1}$, we must have $\xi\perp\VV_{d-1}$. But by assumption $\xi$ is a nonzero element of $\VV_{d-1}$ which is a contradiction. This proves \eqref{eqn:2}.

Now for $x\in\RR^{d-1}$, letting  $\hat{x}=\left(\begin{smallmatrix}x\\0\end{smallmatrix}\right)\in\RR^d$,
\begin{align*}
x'\alpha(t)x
&= \hat{x}'Q'a(t)Q\hat{x}
= (Q\hat{x})'a(t)(Q\hat{x}).
\end{align*}
Since $\mathbf{1}=\sqrt{d}q_d$ and $\hat x_d=0$,
\begin{align*}
Q\hat{x}\cdot\mathbf{1}=(q_1\hat{x}_1+\cdots+q_d\hat{x}_d)\cdot\mathbf{1}=(q_1x_1+\cdots+q_{d-1}x_{d-1})\cdot\mathbf{1}=0. 
\end{align*}
Thus $y=Q\hat{x}\in\VV_{d-1}$, and consequently for $t\in[0,T]$,
\begin{align*}
y'a(t)y
&= \sum_{k\in\mathbf{K}}\sum_{\nu\in\Del^k}(\Gam^k(\mu(t),\nu))y'e_\nu e_\nu'y\\
&\geq \sum_{\nu\in\Del^*}(\Gam^{k(\nu)}(\mu(t),\nu))y'e_\nu e_\nu'y
\geq \kappa(T)y'Gy
\geq \kappa(T)C_G\|y\|^2.
\end{align*}
Thus
\begin{align}\label{eqn:25}
x'\alpha(t)x
\geq \kappa(T)C_G\|Q\hat{x}\|^2=\kappa(T)C_G\|\hat{x}\|^2=\kappa(T)C_G\|x\|^2
\end{align}
and the result follows.
\end{proof}
Since $t\mapsto \alpha(t)$ is Lipschitz, it follows from Lemma \ref{lem:nondeg} that under Condition \ref{con:8}, $t\mapsto\alpha^{1/2}(t)$ is Lipschitz as well (see Theorem 5.2.2 in \cite{stroock2007multidimensional}). Note from \eqref{eqn:25}, that $x'\alpha^{1/2}(t)x\geq (K(T)C_G)^{1/2}\|x\|^2$ for all $x\in\RR^{d\times d}$ and $t\in[0,T]$. In particular 
\begin{align}\label{eqn:26}
\sup_{0\leq t\leq T}\|\alpha^{-1/2}(t)\|<\iy .
\end{align}

\subsection{Main Result}\label{MainRes}
We now present the main result of this work. In Section \ref{Feedback} we will show that for every measurable function $g:[0,T]\times\RR^d\to\Lambda$ there exists a system $\Xi$ and a $U_g\in\cla(\Xi)$ such that the corresponding controlled diffusion process is a (time inhomogeneous) Markov process with generator
\begin{align}\label{eqn:ggen}
\cll_g f(t,x) \doteq \nabla f(x)\cdot[\eta(t,g(t,x))+\beta(t)x]+\frac{1}{2}\Tr (\sigma(t)D^2f(x)\sigma'(t)),\quad f\in\CC^\iy_c(\RR^d)
\end{align}
where $\nabla$ and $D^2$ are the gradient and the Hessian operators, respectively. Furthermore, as we will describe in Section \ref{Feedback}, such a $g$ also defines a control $U^N_g$ in the $N$-th system, under which the state process $\mu_N^g$ is a time inhomogeneous Markov process (see \eqref{eqn:congen}). We refer to $U_g$ and $U^N_g$ as the feedback controls associated with $g$ for the diffusion control problem and the $N$-th controlled system, respectively. The following is the main result of this work. It says the following three things: (i) The value functions of the $N$-particle control problem converge to that of the diffusion control problem as $N\to\iy$; (ii) For every $\eps>0$, there exists a continuous $\eps$-optimal feedback control for the diffusion control problem;  (iii) A near optimal continuous feedback control for the diffusion control problem can be used to construct a sequence of asymptotically near optimal controls for the systems indexed by $N$.
\begin{theorem}\label{thm:main}
Suppose Conditions \ref{con:6}, \ref{con:7}, and \ref{con:8} hold. Let $x_N\in\cls_N$ be such that $v_N=\sqrt{N}(x_N-x_0)\to v_0$ as $N\to\iy$. Then
\begin{enumerate}
\item[(i)] $R_N(v_N)\to R(v_0)$ as $N\to\iy$.
\item[(ii)] For every $\eps>0$, there is a continuous $g_{\eps}:[0,T]\times\RR^d\to\Lambda$ 
such that
\begin{align*}
J(U_{g_{\eps}},v_0)\leq R(v_0)+\eps.
\end{align*}
\item[(iii)] For any continuous $g:[0,T]\times\RR^d\to\Lambda$, $J_N(U^N_g,v_N)\to J(U_g,v_0)$ as $N\to\iy$. In particular, with $g_{\eps}$ as in $(ii)$,
\begin{align*}
R(v_0)=\lim_{N\to\iy}R_N(v_N)\leq \lim_{N\to\iy}J_N(U^N_{g_{\eps}},v_N)\leq R(v_0)+\eps.
\end{align*}
\end{enumerate}
\end{theorem}
\begin{proof}
The above result will be proved in three parts. First in Theorem \ref{thm:6} we will show that for all $v_N,v_0$ as in the statement,
\begin{align*}
\liminf_{N\to\iy}R_N(v_N)\geq R(v_0).
\end{align*}
Next, Theorem \ref{thm:3} shows the first statement in $(iii)$.
Finally in Theorem \ref{thm:5} we 
prove part $(ii)$ of the theorem.

Combining the above results we see that for each $\eps>0$
\begin{align*}
\limsup_{N\to\iy}R_N(v_N)\leq \lim_{N\to\iy}J(U^N_{g_\eps},v_N)=J(U_{g_\eps},v_0)\leq R(v_0)+\eps.
\end{align*}
Since $\eps>0$ is arbitrary it follows immediately that $\limsup_{N\to\iy}R_N(v_N)\leq R(v_0)$
completing the proof of  part $(i)$ and also the second statement in $(iii)$.
\end{proof}
 
Proof of Theorems \ref{thm:6}, \ref{thm:3}, and \ref{thm:5} are given in Sections \ref{LBnd}, \ref{Feedback}, and \ref{NearOptConCont}, respectively. 
Section \ref{Ex1} of the paper will present an example that is a controlled analogue of systems introduced in \cite{antunes2008stochastic} as models for ad hoc wireless networks. We will verify Conditions \ref{con:6}-\ref{con:8} for this example and describe how results from Theorem \ref{thm:main} can be used to construct a sequence of asymptotically near optimal control policies.

\section{Tightness}\label{Tight}

In this section we prove a tightness result which will be needed in the proofs of Theorems \ref{thm:7} and \ref{thm:3}. For $u \in \Lambda_N$, $k \in \mathbf{K}$ and $\nu \in \Delta^k$, we extend the map
$r \to \Gamma_N^k(r,u,\nu)$ to all of $\RR^d$ by setting $\Gamma_N^k(r,u,\nu)=0$ if $r \not \in \cls_N$.

For $U^N\in\cla_N$ define $V_N$ by \eqref{eqn:veen}
where $\mu_N$ is the controlled Markov process corresponding to the system under control $U^N$ given as in \eqref{eqn:stpro}. Define $\gamma_N:[0,T]\times\RR^d\to\RR^d$ as $\gamma_N(t,x)\doteq \mu(t)+\frac{1}{\sqrt{N}}x$, for $x\in\RR^d,\ t\in[0,T]$ and for $\phi\in \CC^2(\RR^d),\ s\in[0,T],\ u\in\Lambda_N$, and $y\in \RR^d$ define
\begin{align}\label{eqn:10}
\mathcal{L}_u^N(\phi,s,y)
&\doteq \sum_{k\in\mathbf{K}}\sum_{\nu\in\Del^k}\Gam_N^{k}\left(\gamma_N(s,y),u,\nu\right)\left[\phi\left(y+\frac{1}{\sqrt{N}}e_\nu\right)-\phi(y)\right]-\sqrt{N}F(\mu(s))\nabla\phi(y).
\end{align}
For $i=1,\ldots,d$ define $\phi^i(y)\doteq y_i$ and denote the $i$-th coordinate of $e_\nu$ and $F$ by $e_\nu^i$ and $F^i$ respectively. Let
\begin{align*}
b_N^{i,u}(s,y)
&\doteq \mathcal{L}_u^N(\phi^i,s,y)
= \sum_{k\in\mathbf{K}}\sum_{\nu\in\Del^k}\Gam_N^{k}\left(\gamma_N(s,y),u,\nu\right)\frac{1}{\sqrt{N}}e_\nu^i-\sqrt{N}F^i(\mu(s))\\
&= \sqrt{N}\sum_{k\in\mathbf{K}}\sum_{\nu\in\Del^k}e_\nu^i\left(\frac{1}{N}\Gam_N^{k}\left(\gamma_N(s,y),u,\nu\right)-\Gam^k(\mu(s),\nu)\right)
\end{align*}
where the second equality follows from the definition of $F$ in Proposition \ref{prop:2}. Also, for $i,j=1,\ldots,d$ let,
\begin{align*}
a_N^{i,j,u}(s,y)
&\doteq \mathcal{L}_u^N(\phi^i\phi^j,s,y)-y_ib_N^{j,u}(s,y)-y_jb_N^{i,u}(s,y)\\
&= \sum_{k\in\mathbf{K}}\sum_{\nu\in\Del^k}\Gam_N^{k}\left(\gamma_N(s,y),u,\nu\right)\left(\frac{y_i}{\sqrt{N}}e_\nu^j+\frac{y_j}{\sqrt{N}}e_\nu^i+\frac{1}{N}e_\nu^ie_\nu^j\right)\\
&\qquad-y_i\sqrt{N}F^j(\mu(s))-y_j\sqrt{N}F^i(\mu(s))-y_ib_N^{j,u}(s,y)-y_jb_N^{i,u}(s,y)\\
&= \sum_{k\in\mathbf{K}}\sum_{\nu\in\Del^k}\Gam_N^{k,u}\left(\gamma_N(s,y)),u,\nu\right)\frac{1}{N}e_\nu^ie_\nu^j.
\end{align*}
We write $b_N^u=(b_N^{1,u},\ldots,b_N^{d,u})$ and $a_N^u=(a_N^{i,j,u})_{i,j=1,\ldots,d}$. 

Let
\begin{align}\label{eqn:11}
n_\mathbf{K}\doteq2\max_{k\in\mathbf{K}}n_k.
\end{align}
The following Lemma gives a key bound needed for tightness.
\begin{lemma}\label{H1}
Suppose Condition \ref{con:6} holds. Then there exists $C_3\in(0,\iy)$ such that for every $N\in \NN$ and $t\in [0,T]$ 
\begin{align*}
(\|b_N^{U^N(t)}(t,V_N(t))\|^2+\Tr (a_N^{U^N(t)}(t,V_N(t)))\leq C_3(1+\|V_N(t)\|^2)
\end{align*}
almost everywhere for every $U^N\in\cla_N$.
\end{lemma}
\begin{proof}
It follows from \eqref{eqn:1} that for $y\in B(2\sqrt{N})$ such that $\mu(t)\in\cls_N(y)$, $u\in\Lambda_N$, and $i=1,\ldots,d$
\begin{align*}
a_N^{i,i,u}(t,y)
&= \sum_{k\in\mathbf{K}}\sum_{\nu\in\Del^k}\Gam_N^{k}\left(\gamma_N(t,y),u,\nu\right)\frac{1}{N}e_\nu^ie_\nu^i
\leq \sum_{k\in\mathbf{K}}\sum_{\nu\in\Del^k}C_2e_\nu^ie_\nu^i
\leq C_2\mathbf{\ell} n_\mathbf{K}^2,
\end{align*} 
and from Condition \ref{con:6}
\begin{align*}
b_N^{i,u}(t,y)^2
&= \left(\sum_{k\in\mathbf{K}}\sum_{\nu\in\Del^k}e_\nu^i\sqrt{N}\left(\frac{1}{N}\Gam_N^{k}\left(\gamma_N(t,y),u,\nu\right)-\Gam^k(\mu(t),\nu)\right)\right)^2\\
&\leq \left( \sum_{k\in\mathbf{K}}\sum_{\nu\in\Del^k}|e^i_\nu| C_1(1+\|y\|)\right)^2\\
&\leq ( C_1\mathbf{\ell}n_\mathbf{K}(1+\|y\|))^2\\
&\leq 2C_1^2\mathbf{\ell}^2n_\mathbf{K}^2(1+\|y\|^2).
\end{align*}
The result now follows on noting that $V_N(t)\in B(2\sqrt{N})$ and $\mu(t)\in\cls_N(V_N(t))$ a.s.
\end{proof}
For $N\geq 1$ and $\phi\in\CC^2(\RR^d)$, let $\psi_N\in\CC^{1,2}([0,T]\times\RR^d)$ be defined as
\begin{align*}
\psi_N(t,y)\doteq \phi(\sqrt{N}(y-\mu(t))),\ t\in[0,T],\ y\in\RR^d.
\end{align*}
Note that $\phi(x)=\psi_N(t,\gamma_N(t,x))$. Using \eqref{eqn:stpro} and Dynkin's formula,
\begin{align}\label{eqn:gen1}
\begin{split}
\phi(V_N(t))
&= \psi_N(t,\mu_N(t))\\
&\qquad= \psi_N(0,\mu_N(0))+\int_0^tL_{U^N(s)}^N\psi_N(s,\mu_N(s))ds+\int_0^t\frac{\partial}{\partial s}\psi_N(s,\mu_N(s))ds + M_t^{N,\phi}
\end{split}
\end{align}
where $M^{N,\phi}_t$ is a locally square-integrable martingale and for $u\in\Lambda_N,(s,r)\in[0,T]\times\RR^d$,
\begin{align*}
L_u^N\psi_N(s,r)
&\doteq \sum_{k\in\mathbf{K}}\sum_{\nu\in\Del^k}\Gam_N^{k}(r,u,\nu)\left[\psi_N\left(s,r+\frac{1}{N}e_\nu\right)-\psi_N(s,r)\right]\\
&=\sum_{k\in\mathbf{K}}\sum_{\nu\in\Del^k}\Gam_N^{k}(r,u,\nu)\left[\phi\left(\sqrt{N}(r-\mu(s))+\frac{1}{\sqrt{N}}e_\nu\right)-\phi(\sqrt{N}(r-\mu(s)))\right].
\end{align*}
Also, since $\dot{\mu}(t)=F(\mu(t))$,
\begin{align*}
\frac{\partial}{\partial s}\psi_N(s,r)
&= -\sqrt{N}F(\mu(s))\cdot\nabla\phi(\sqrt{N}(r-\mu(s))).
\end{align*}
This shows that the process $V_N$ is a $\cld$-semimartingale in the sense of Definition 3.1.1 of \cite{joffe1986weak} with increasing function $A(t)=t$ and the associated mapping
$\mathbf{L}^N: \CC^2(\RR^d)\times \RR^d \times [0,T] \times \Omega^N \to \RR$ (in the notation of \cite{joffe1986weak}) defined as
$$\mathbf{L}^N(\phi, y, t, \omega) \doteq \cll^N_{U^N(t,\omega)}(\phi, t,y),$$
where $\cll^N_u$ is defined as in \eqref{eqn:10}.
Furthermore,
$$\mathbf{b}^N_i(y,t,\omega) \doteq b^{i, U^N(t,\omega)}(t,y), \;\;
\mathbf{a}^N_{ij}(y,t,\omega) \doteq a_N^{i, U^N(t,\omega)}(t,y),$$
are the local coefficients of first and second order of the semimartingale $V_N$ in the sense of Definition 
3.1.2 of \cite{joffe1986weak}. In particular,
equation \eqref{eqn:gen1} combined with \eqref{eqn:10} implies that
\begin{align}\label{eqn:tightmart}
M_t^N \doteq V_N(t)-V_N(0)-\int_0^t \mathbf{b}^N(V_N(s),s,\omega)ds
\end{align}
is a $d$-dimensional locally square-integrable martingale. 
\begin{definition}
For $x\in\DD([0,T]:\RR^d)$ let $j_T(x)\doteq\sup_{0< t\leq T}\|x(t)-x(t-)\|$ be the maximum jump size of $x$. We say a tight collection of $\DD([0,T]:\RR^d)$-valued random variables $\{X_N\}_{N\in\NN}$ is $\CC$-tight if $j_T(X_N)\Rightarrow 0$.
\end{definition}
If $X_N,X$ are $\DD([0,T]:\RR^d)$-valued random variables and $X_N\Rightarrow X$ then $\PP(X\in\CC([0,T]:\RR^d))=1$ if and only if $\{X_N\}_{N\in\NN}$ is $\CC$-tight \cite{BillingsleyConv}. Using Lemma \ref{H1}, the following Proposition follows directly from Lemma 3.2.2 and Proposition 3.2.3 of \cite{joffe1986weak}.
\begin{proposition}\label{prop:3}
Suppose Condition \ref{con:6} holds. Define for $N\in\NN,\ V_N$ through \eqref{eqn:veen}, where $\mu_N$ is defined as in \eqref{eqn:stpro} for some $U^N\in\cla_N$. Suppose $V_N(0)=v_N\in\RR^d$ and $\sup_N\|v_N\|<\iy$. Then 
\begin{align*}
\sup_{N\geq 1}\E\sup_{0\le t\leq T}\|V_N(t)\|^2<\iy
\end{align*}
and the sequence $\{V_N\}_{N\geq 1}$ is a tight collection of $\DD([0,T]:\RR^d)$-valued random variables. Furthermore the sequence is $\CC$-tight.
\end{proposition}
\begin{proof}
Since $\mathbf{b}^N$ and $\mathbf{a}^N$ are  the local coefficients of the semimartingale $V_N$, the moment bound is immediate from the properties of $b_N^u$ and $a_N^u$ established in Lemma \ref{H1} upon using Lemma 3.2.2 of \cite{joffe1986weak}. Using this moment bound and Lemma \ref{H1} once again, tightness follows from verifying Aldous' tightness criteria (c.f. Theorem 2.2.2 in \cite{joffe1986weak}) as in Proposition 3.2.3 of \cite{joffe1986weak}. Also note that $\{V_N\}$ is $\CC$-tight because $j_T(V_N)\leq\frac{1}{\sqrt{N}}\mathbf{\ell}^{1/2}n_{\mathbf{K}}$  where $\mathbf{\ell}$ and $n_{\mathbf{K}}$ are as in \eqref{eqn:15} and \eqref{eqn:11}, respectively.
\end{proof}

\begin{remark}
 Proposition \ref{prop:3} in particular says that under Condition \ref{con:6} $\mu_N$ converges to $\mu$ in $\DD([0,T]:\RR^d)$.
\end{remark}

\section{Lower Bound}\label{LBnd}
In this section we prove the following result.
\begin{theorem}\label{thm:6}
Suppose Conditions \ref{con:6}, \ref{con:7}, and \ref{con:8} hold. Let $v_n,v_0$ be as in the statement of Theorem \ref{thm:main}. Then
\begin{align*}
\liminf_{N\to\iy}R_N(v_N)\geq R(v_0).
\end{align*}

\end{theorem}

We first note that the local martingale $M^N$ in \eqref{eqn:tightmart} takes the following explicit form.
\begin{align}\label{eqn:27}
M^N(t)=\sqrt{N}\sum_{k\in\mathbf{K}}\sum_{\nu\in\Del^k}e_\nu M_{k,\nu}^N(t),\; t \in [0,T],
\end{align}
where $M^N_{k,\nu}$ is as defined in \eqref{eqn:mart}. Indeed, denoting the right side
of \eqref{eqn:27} as $\tilde M^N(t)$ and using \eqref{eqn:stpro} we can write,
\begin{align*}
\mu_N(t)=\mu_N(0)+\frac{1}{N}\sum_{k\in\mathbf{K}}\sum_{\nu\in\Del^k}e_\nu\int_0^t\Gam_N^k(\mu_N(s),U^N(s),\nu)ds+\frac{1}{\sqrt{N}}\tilde M^N(t).
\end{align*}
From this and recalling the definition of $\mu$ from \eqref{eqn:stpro} and of $H^k$ from Condition \ref{con:7}, we have
the following representation for $V_N$ in terms of $\tilde M^N$ 
\begin{align}\label{eqn:vmart}
\begin{split}
V_N(t)
&= \sqrt{N}(\mu_N(t)-\mu(t))\\
&= v_N+\sum_{k\in\mathbf{K}}\sum_{\nu\in\Del^k}e_\nu\int_0^t\sqrt{N}\left(\frac{1}{N}\Gam_N^k(\mu_N(s),U^N(s),\nu)-\Gam^k(\mu(s),\nu)\right)ds+\tilde M^N(t)\\
&= v_N+\sum_{k\in\mathbf{K}}\sum_{\nu\in\Del^k}e_\nu\int_0^tH^k(V_N(s),\mu(s),\sqrt{N}U^N(s),\nu)ds+\int_0^t\vartheta^N(s)ds+\tilde M^N(t)
\end{split}
\end{align}
where the error term $\vartheta^N$ is given as
\begin{align*}
\vartheta^N(s)=\sum_{k\in\mathbf{K}}\sum_{\nu\in\Del^k}\vartheta^N_{k,\nu}(s),\;\;
\vartheta^N_{k,\nu}(s)= e_\nu\beta_k^N(V_N(s),\mu(s),\sqrt{N}U^N(s),\nu),
\end{align*}
and $\beta_k^N$ is as in Condition \ref{con:7}. 
This proves \eqref{eqn:27}.

Note that $\vartheta^N$ can be estimated as
\begin{align}\label{eqn:13}
\|\vartheta^N(s)\|\leq \theta^N(V_N(s)).
\end{align}
where for $y\in\RR^d$
\begin{align*}
\theta^N(y)\doteq (\mathbf{\ell})^{1/2} n_{\mathbf{K}}\sup_{\xi\in\cls_N(y)}\sup_{u\in\Lambda}\sum_{k\in\mathbf{K}}\sum_{\nu\in\Del^k}|\beta_k^N(y,\xi,u,\nu)|,
\end{align*}
with $\mathbf{\ell}$ and $n_{\mathbf{K}}$ as in \eqref{eqn:15} and \eqref{eqn:11}, respectively.
Condition \ref{con:7} then implies
\begin{align}\label{eqn:14}
\sup_{y\in A}\theta^N(y)&\to0,\text{ as }N\to\iy
\end{align}
for all compact $A$. The above estimate will allow us to estimate the error term $\vartheta^N$ in \eqref{eqn:vmart}.

In order to have suitable tightness properties of the control processes it will be convenient to introduce the following collection of random measures. Define $\clm([0,T]\times \Lambda)$ valued random variables
 $m^N$ as
\begin{align}\label{eqn:relcon}
m^N(A\times B)=\int_A1_B(\sqrt{N}U^N(s))ds.
\end{align}
Note that $m^N$ can be disintegrated as $m_s^N(du)ds$, where $m^N_s(du)=\del_{\sqrt{N}U^N(s)}(du)$ and $\del_x$ is the Dirac measure at the point $x$. Then for $s\in[0,T]$,
\begin{align*}
H^k(V_N(s),\mu(s),\sqrt{N}U^N(s),\nu)
&= \int_\Lambda h_1^k(\nu,\mu(s))u_{k,\nu}m_s^N(du)+ h_2^k(\nu,\mu(s))\cdot V_N(s).
\end{align*}
Thus the state equation \eqref{eqn:vmart} can be rewritten as
\begin{align}\label{eqn:vmart2}
\begin{split}
V_N(t)
&= v_N+\int_0^t\vartheta^N(s)ds+M^N(t) + \sum_{k\in\mathbf{K}}\sum_{\nu\in\Del^k}e_\nu\int_0^t\int_\Lambda h_1^k(\nu,\mu(s))u_{k,\nu}m_s^N(du)ds\\
&\qquad+\sum_{k\in\mathbf{K}}\sum_{\nu\in\Del^k}e_\nu\int_0^t h_2^k(\nu,\mu(s))\cdot V_N(s)ds.
\end{split}
\end{align}
Recall from Section \ref{Notation} that $\clm\left([0,T]\times\Lambda\right)$ is the space of all  finite measures on $[0,T]\times\Lambda$ equipped with the usual weak convergence topology. 
\begin{theorem}\label{thm:7}
Suppose Conditions \ref{con:6}, \ref{con:7}, and \ref{con:8} hold and let $v_N,v_0$ be as in Theorem \ref{thm:main}. Then:
\begin{enumerate}
\item[(i)] $Y^N=\{V_N,M^N,m^N,\int_0^\cdot\vartheta^N(s)ds\}_{N\geq 1}$ is a tight collection of $\DD([0,T]:\RR^{2d})\times\clm([0,T]\times\Lambda)\times\CC([0,T]:\RR^d)$ valued random variables.
\item[(ii)]  $\int_0^\cdot\vartheta^N(s)ds$ converges to 0 in probability in $\CC([0,T]:\RR^d)$.
\item[(iii)] $(V_N,M^N)_{N\geq 1}$ is $\CC$-tight.
\item[(iv)] Suppose $\{Y^N\}$ converges weakly along a subsequence to $Y=(V,M,m,0)$ defined on a probability space $(\Om^*,\clf^*,\PP^*)$. Then, $\PP^*$ a.s., the first marginal of $m$ is the Lebesgue measure on $[0,T]$. Disintegrating $m$ as 
\begin{align*}
m(A\times B)=\int_Am_t(B)dt,\quad A\in\clb([0,T]),\ B\in\clb(\Lambda),
\end{align*}
define
\begin{align}\label{eqn:control}
U_{k,\nu}(t)\doteq\int_\Lambda u_{k,\nu}m_t(du),\quad t\in [0,T],\ k\in\mathbf{K},\ \nu\in\Del^k.
\end{align}
Let $\{B_d(t)\}$ be a one dimensional standard Brownian motion given on $(\Om^*,\clf^*,\PP^*)$ that is independent of $Y$. Let $\clg_t^\circ=\sigma\{B_d(s),V(s),M(s),m([0,s]\times B):s\leq t,B\in\clb(\Lambda)\}$ and $\clg_t$ be the $\PP^*$-completion of $\clg_t^\circ$. Then there is a $d$-dimensional $\{\clg_t\}$-Brownian motion $\{W(t)\}, W=(W_1,\ldots,W_d)$ such that the following equation is satisfied
\begin{align}\label{eqn:9}
\begin{split}
V(t)&= v_0+\int_0^t\sigma(s)dW(s) +\sum_{k\in\mathbf{K}}\sum_{\nu\in\Del^k}e_\nu\int_0^t\int_\Lambda h_1^k(\nu,\mu(s))U_{k,\nu}(s)ds\\
&\qquad+\sum_{k\in\mathbf{K}}\sum_{\nu\in\Del^k}e_\nu\int_0^t h_2^k(\nu,\mu(s))\cdot V(s)ds\\
&= v_0 +\int_0^t\eta(s,U(s))ds+\int_0^t\beta(s)V(s)ds+\int_0^t\sigma(s)dW(s).
\end{split}
\end{align}
\end{enumerate}
\end{theorem}
\begin{proof}
Tightness of $\{m^N\}$ as $\clm([0,T]\times\Lambda)$-valued random variables is immediate since $m^N([0,T]\times\Lambda)=T$ for all $N$ and $\Lambda$ is a compact set. $\CC$-tightness of $\{V_N\}$ was proved in Proposition \ref{prop:3}. 

In order to verify the tightness of $\{M^N\}_{N\geq1}$, we will use Theorem 2.3.2 of \cite{joffe1986weak} (see Theorem \ref{thm:apptight} in Appendix). According to this theorem it suffices to verify conditions [A] and [T$_1$], given in Theorem \ref{thm:apptight}, for the sequence of quadratic variation processes, $\{\sum_{k\in\mathbf{K}}\sum_{\nu\in\Del^k}N\lan M^N_{k,\nu}\ran\}_{N\geq 1}$. Note that
\begin{align*}
\sum_{k\in\mathbf{K}}\sum_{\nu\Del^k}N\lan M^N_{k,\nu}\ran(t)
&= \frac{1}{N}\sum_{k\in\mathbf{K}}\sum_{\nu\in\Del^k}\int_0^t\Gam_N^k(\mu_N(s),U^N(s),\nu)ds.
\end{align*}
Condition [A] and [T$_1$] are now immediate on noting that Condition \ref{con:6} implies (see \eqref{eqn:1})
\begin{align*}
\frac{1}{N}\Gam_N^k(\mu_N(s),U^N(s),\nu)\leq C_2
\end{align*}
almost surely for all $N,k,\nu$, and $s$. Furthermore $\{M^N\}$ is $\CC$-tight because $j_T(M^N)\leq\frac{1}{\sqrt{N}}\mathbf{\ell}^{1/2}n_{\mathbf{K}}$.

Finally, from \eqref{eqn:13}, for $\del>0$ we have that
\begin{align*}
\PP\left[\sup_{0\leq s\leq T}\left\|\int_0^s\vartheta^N(u)du\right\|>\del\right]
&\leq\PP\left[\sup_{0\leq s\leq T}\theta^N(V_N(s))>\frac{\del}{T}\right].
\end{align*}
Since $\{V_N\}$ is $\CC$-tight for every $\eps>0$, there exists some $\kappa_1<\iy$ such that 
\begin{align*}
\PP\left[\sup_{0\leq s\leq T}\|V_N(s)\|> \kappa_1\right]\leq\eps
\end{align*}
for all $N\in\NN$. Recalling \eqref{eqn:14} we see that there exists an $N_0>0$ such that 
\begin{align*}
\sup_{y:\|y\|\leq\kappa_1}\theta^N(y)\le\frac{\del}{T}
\end{align*}
for all $N\geq N_0$. Thus for all $N\geq N_0$
\begin{align*}
&\PP\left[\sup_{0\leq s\leq T}\left\|\int_0^s\vartheta^N(u)du\right\|>\del\right]\\
&\qquad\leq\PP\left[\sup_{0\leq s\leq T}\theta^N(V_N(s))>\frac{\del}{T},\sup_{0\leq s\leq T}\|V_N(s)\|\leq \kappa_1\right]+\PP\left[\sup_{0\leq s\leq T}\|V_N(s)\|>\kappa_1 \right]
\leq \eps.
\end{align*}
Since $\eps>0$ is arbitrary we conclude that $\{\int_0^\cdot\vartheta^N(s)ds\}$ converges to 0 in probability in $\CC([0,T]:\RR^d)$. This concludes the proof of $(i),(ii)$ and $(iii)$.

Consider now $(iv)$. Let $Y$ be as in the statement of the theorem, namely $Y^N$ converges weakly along a subsequence to $Y=(V,M,m,0)$. The property that the last component of $Y$ must be 0 is a consequence of $(ii)$. For notational convenience we label the subsequence once more by $\{N\}$. Recall the orthogonal matrix $Q=[q_1\ q_2\ \ldots\ q_d]$ and function $a:[0,T]\to\RR^{d\times d}$ defined in Section \ref{DiffCont} as well as the function $\alpha^{1/2}:[0,T]\to\RR^{(d-1)\times(d-1)}$ introduced above \eqref{eqn:7}. Define $(d-1)$ and 1 dimensional processes $\hat{M}^N$ and $R^N$, respectively, as
\begin{align}\label{eqn:martmat}
\left(\begin{matrix}\hat{M}^N(t)\\
R^N(t)\end{matrix}\right)=Q'M^N(t).
\end{align}

Note that
\begin{align*}
R^N(t)
&= q_d'M^N(t)
= \sum_{k\in\mathbf{K}}\sum_{\nu\in\Del^k}\frac{1}{\sqrt{d}}\mathbf{1}'e_\nu M^N_{k,\nu}(t)=0
\end{align*}
since $\mathbf{1}'e_\nu=0$ for all $k\in\mathbf{K},\nu\in\Del^k$.

We now show that $M$ is a $\{\clg_t\}$-martingale. The Burkholder-Davis-Gundy inequality (see Theorem IV.48 of \cite{Protter2005}) implies that there exists $\kappa_2\in(0,\iy)$ such that for  $i=1,\ldots,d$
\begin{align}\label{eqn:28}
\begin{split}
&\sup_{N\in\NN}\E\sup_{0\leq t\leq T}(M_i^N)^4
\leq\sup_{N\in\NN} \kappa_2N^2\sum_{k\in\mathbf{K}}\sum_{\nu\in\Del^k}\E[M^N_{k,\nu}]_T^2\\
&\qquad = \sup_{N\in\NN}\kappa_2\sum_{k\in\mathbf{K}}\sum_{\nu\in\Del^k}\E\left(\frac{1}{N}\mathcal{N}_{k,\nu}
\left(\int_0^T\Gam^k_N(\mu_N(s),U^N(s),\nu)ds\right)\right)^2\\
&\qquad \leq\sup_{N\in\NN}\kappa_2\sum_{k\in\mathbf{K}}\sum_{\nu\in\Del^k}\E\left(\frac{1}{N}\mathcal{N}_{k,\nu}(NTC_2)\right)^2
<\iy ,
\end{split}
\end{align}
where the first inequality on the last line is from \eqref{eqn:1}. Thus $\{\sup_{0\leq t\leq T}\|M^N(t)\|^2\}_{N\geq 1}$ is uniformly integrable. Let $k\in\NN$ and $\clh:(\RR^d\times\RR^d\times\RR)^k\to \RR$ be a bounded and continuous function. For $0\le s\leq t\leq T$ and $0\leq t_1\leq \ldots\leq t_k\leq s$ we let $\xi^N_i=(V_N(t_i),M^N(t_i),m^N_i(f))$ and $\xi_i=(V(t_i),M(t_i),m_i(f))$ where $m^N_i(f)=\int_{\Lambda\times[0,t_i]}f(u)m_s^N(du)ds$, $m_i(f)=\int_{\Lambda\times[0,t_i]}f(u)m_s(du)ds$ and $f\in\CC_b(\Lambda)$. Then
\begin{align*}
\E^*\clh(\xi_1,\ldots,\xi_k)[M(t)-M(s)]
&= \lim_{N\to\iy}\E\clh(\xi_1^N,\ldots,\xi_k^N)[M^N(t)-M^N(s)]
=0
\end{align*}
where the first equality follows from the uniform integrability property noted above, and the second equality is a consequence of the martingale property of $M^N$ (which is a consequence of \eqref{eqn:28}). Combining this with the fact that $B_d$ is a Brownian motion independent of $Y$, it follows that $M$ is a $\{\clg_t\}$-martingale.

We now define the process which will converge to the Brownian motion driving the limit diffusion. Recall that the matrix $\alpha^{1/2}$ is invertible and the property \eqref{eqn:26}. Define the $(d-1)$-dimensional processes $B^N(t)=(B_i^N(t))_{i=1}^{d-1}$ as
\begin{align*}
B^N_i(t)=\sum_{j=1}^{d-1}\int_0^t\alpha^{-1/2}_{ij}(s)d\hat{M}_j^N(s),
\end{align*}
where $\hat{M}^N$ is as in \eqref{eqn:martmat}. Since $M^N$ is a $\{\clf^N_t\}$-martingale, both $\hat{M}^N$ and $B^N$ are $\{\clf^N_t\}$-martingales as well. From the estimate in \eqref{eqn:28} it follows that $\{\sup_{0\leq t\leq T}\|B^N(t)\|^2\}_{N\geq 1}$ is uniformly integrable. Also note that for integers $1\leq i,j\leq d-1$, the cross quadratic variation of $B_i^N$ and $B_j^N$ can be expressed as
\begin{align*}
\lan B^N_i,B^N_j\ran(t)
&= \sum_{m_1=1}^{d-1}\sum_{m_2=1}^{d-1}\int_0^t\alpha^{-1/2}_{im_1}(s)\alpha^{-1/2}_{jm_2}(s)d\lan\hat{M}^N_{m_1},\hat{M}^N_{m_2}\ran(s).
\end{align*}
Note that for all $t\in [0,T]$
\begin{align*}
\lan\hat{M}^N_{m_1},\hat{M}^N_{m_2}\rangle(t)
&= \lan q_{m_1}'M^N,q_{m_2}'M^N\ran (t)
= \sum_{m_3=1}^{d}\sum_{m_4=1}^{d}q_{m_3m_1}q_{m_4m_2}\lan M^N_{m_3},M^N_{m_4}\ran(t)
\end{align*}
where
\begin{align*}
\lan M^N_{m_3},M^N_{m_4}\ran(t)
&= \sum_{k\in\mathbf{K}}\sum_{\nu\in\Del^k}e_\nu^{m_3}e_\nu^{m_4}\frac{1}{N}\int_0^t\Gam_N^k(\mu_N(s),U^N(s),\nu)ds.
\end{align*}
Thus
\begin{align}\label{eqn:crossvar}
\begin{split}
\lan B^N_i,B^N_j\ran(t)
&=\int_0^t\left(Q'\sigma(s)^{-1}\sum_{k\in\mathbf{K}}\sum_{\nu\in\Del^k}\frac{1}{N}\left(\Gam_N^k(\mu_N(s),U^N(s),\nu)e_\nu e_\nu'\right)(\sigma(s)')^{-1}Q\right)_{ij}ds\\
&=\int_0^t\left(Q'\sigma(s)^{-1}a(s)(\sigma(s)')^{-1}Q\right)_{ij}ds+\eps^N_{ij}(t)
= tI_{ij}+\eps^N_{ij}(t)
\end{split}
\end{align}
where $I$ is the $d\times d$ identity matrix,
\begin{align*}
\eps^N(t)
&=\int_0^tQ'\sigma(s)^{-1}\sum_{k\in\mathbf{K}}\sum_{\nu\in\Del^k}\left(\frac{1}{N}\Gam_N^k(\mu_N(s),U^N(s),\nu)-\Gam^k(\mu(s),\nu)\right)e_\nu e_\nu'(\sigma(s)')^{-1}Qds
\end{align*}
and $\eps^N_{ij}$ is the $(i,j)$-th coordinate of $\eps^N$. From Condition \ref{con:6} and \eqref{eqn:26} we have that $\E\|\eps^N(t)\|\to0$ for all $t$ as $N\to\iy$.

Also it is easy to see that (cf. Theorem 2.2 of \cite{kurtz1991weak})  $B^N(\cdot)\Rightarrow\int_0^\cdot\alpha^{-1/2}(s)d\hat{M}(s)\doteq B(\cdot)$ in $\DD([0,T]:\RR^{d-1})$, where
 $\left(\begin{smallmatrix}\hat{M} \\ 0\end{smallmatrix}\right)=Q'M$. Also since $\{\sup_{0\leq t\leq T}\|B^N(t)\|^2\}_{N\geq 1}$ is uniformly integrable, we have from \eqref{eqn:crossvar} that
\begin{align*}
&\E^*\left(\clh(\xi_1,\ldots,\xi_k)[B(t)B'(t)-B(s)B'(s)-(t-s)I]\right)\\
&\qquad= \lim_{N\to\iy}\E\left(\clh(\xi_1^N,\ldots,\xi_k^N)[B^N(t)(B^N)'(t)-B^N(s)(B^N)'(s)-(t-s)I]\right)\\
&\qquad= \lim_{N\to\iy}\E\left(\clh(\xi_1^N,\ldots,\xi_k^N)\eps^N(t)\right)
=0.
\end{align*}
Combining this with the fact that $B_d$ is independent of $Y$ we see that $B$ is a $(d-1)$-dimensional continuous $\clg_t$-martingale with quadratic variation $\lan B\ran(t)=tI$ which implies, by L\'{e}vy's theorem, that $B$ is a $(d-1)$-dimensional $\{\clg_t\}$-Brownian motion. Since $B_d$ is a Brownian motion independent of $Y$, it follows that $\hat{W}\doteq(B, B_d)'$ is a $d$-dimensional $\{\clg_t\}$-Brownian motion. Also note that 
\begin{align}\label{eqn:marthat}
\hat{M}(t)=\int_0^t\alpha^{1/2}(s)dB(s).
\end{align}

The final step of the proof is to show that $V$ is a solution to \eqref{eqn:9} with $W=Q\hat{W}$. Note that since $Q$ is orthogonal, $W$ is a $d$-dimensional $\{\clg_t\}$-Brownian motion as well. From the definition of $\eta$ and since $e_\nu\cdot\mathbf{1}=0$ for all $k\in\mathbf{K},\nu\in\Del^k$, $Q'\eta$ takes the form 
\begin{align}\label{eqn:hatdef1}
Q'\eta(t,u)=\left(\begin{matrix}\hat{\eta}(t,u)\\0\end{matrix}\right).
\end{align}
Similarly, from the expression for $\beta$ and from \eqref{eqn:7} it follows that
\begin{align}\label{eqn:hatdef2}
Q'\beta(t) Q=\left(\begin{matrix}
\hat{\beta}(t) & 0\\
0 & 0
\end{matrix}\right), \; \; Q'\sigma(t)Q=\left[\begin{matrix}
\alpha^{1/2}(t) & 0\\
0 & 0
\end{matrix}\right].
\end{align}
Also since $V_N\cdot\mathbf{1}=0$ and $v_0\cdot\mathbf{1}=0$, we have
\begin{align}\label{eqn:hatdef3} 
Q'V=\left[\begin{matrix}
\hat{V}\\ 0
\end{matrix}\right], \;\; Q'v_0=\left[\begin{smallmatrix}\hat{v}_0\\0\end{smallmatrix}\right].
\end{align}
We first show that $\hat{V}$ solves the $d-1$ dimensional equation
\begin{align}\label{eqn:17}
\hat{V}(t)=\hat{v}_0+\int_0^t\hat{\eta}(s,u)m_s(du)ds+\int_0^t\hat{\beta}(s)\hat{V}(s)ds+\int_0^t\alpha^{1/2}(s)dB(s).
\end{align} 
Letting $\left[\begin{smallmatrix}\hat{V}_N\\0\end{smallmatrix}\right]\doteq Q'V_N$ and using \eqref{eqn:vmart2}, we have,
\begin{align*}
\hat{V}_N(t)=\hat{v}_N+\int_0^t\hat{\eta}(s,u)m_s^N(du)ds+\int_0^t\hat{\beta}(s)\hat{V}_N(s)ds+\int_0^t\hat{\vartheta}^N(s)ds+\hat{M}^N(t)
\end{align*}
where $\left[\begin{smallmatrix}\hat{v}_N\\0\end{smallmatrix}\right]=Q'v_N$ and $\left[\begin{smallmatrix}\hat{\vartheta}^N\\0\end{smallmatrix}\right]=Q'\vartheta^N$.  Note that $(\hat{V}_N,\hat{M}^N,m^N,\hat{\vartheta}^N)\Rightarrow(\hat{V},\hat{M},m,0)$. Without loss of generality we  assume that the convergence holds a.s.

Since $m^N\to m$, we have
\begin{align*}
\int_0^t\int_\Lambda h_1^k(\nu,\mu(s))u_{k,\nu}m_s^N(du)ds
\to\int_0^t\int_\Lambda h_1^k(\nu,\mu(s))u_{k,\nu}m_s(du)ds
\end{align*}
and thus
\begin{align}\label{eqn:18}
\int_0^t\hat{\eta}(s,u)m_s^N(du)ds\to\int_0^t\hat{\eta}(s,u)m_s(du)ds.
\end{align}
Similarly it follows that
\begin{align}\label{eqn:19}
\int_0^t\hat{\beta}(s)\cdot \hat{V}_N(s)ds
\to \int_0^t \hat{\beta}(s)\cdot \hat{V}(s)ds.
\end{align}
Combining \eqref{eqn:18} and \eqref{eqn:19} with \eqref{eqn:marthat} we see that $\hat{V}$ satisfies \eqref{eqn:17}. Recalling the relation between $(\hat{v}_0,\hat{V},\hat{\eta},\hat{\beta},\alpha^{1/2})$ and $(v_0,V,\eta,\beta,\sigma)$ we see that $V=Q\left[\begin{smallmatrix}\hat{V}\\0\end{smallmatrix}\right]$ is a solution of \eqref{eqn:6}, where $W=Q\hat{W}$. This proves $(iv)$ and thus completes the proof of the theorem.
\end{proof}

We now apply the above result to prove Theorem \ref{thm:6} which shows that the limit of the value of the optimal control problem for the $N$-th system as $N\to\iy$ can be bounded from below by the value of the control problem for the limit diffusion.

\begin{proof}[Proof of Theorem \ref{thm:6}]
Let $v_N,v_0$ be as in the statement of the theorem. It suffices to show that for any sequence of admissible controls $\{U^N\}$,  $\liminf_{N\to\iy}J_N(U^N,v_N)\geq R(v_0)$. Let $U^N\in\cla_N$, and $m^N$ be the corresponding relaxed control defined as in \eqref{eqn:relcon}. From the previous theorem we have that $\{(V_N,M^N,m^N,\int_0^\cdot\vartheta^N(s)ds)\}_{N\geq 1}$ is tight and thus every subsequence (also denoted with the index $N$) has a further subsequence $\{(V_{N_\ell},M^{N_\ell},m^{N_\ell},\int_0^\cdot\vartheta^{N_\ell}(s)ds)\}$ such that 
\begin{align*}
(V_{N_\ell},M^{N_\ell},m^{N_\ell},\int_0^\cdot\vartheta^{N_\ell}(s)ds) \Rightarrow (V,M,m,0).
\end{align*}
Furthermore, equation \eqref{eqn:9} holds for the limit point $(V,M,m,0)$ with a $\{\clg_t\}$-Brownian motion $W$ where $\{\clg_t\}$ is as in the statement of Theorem \ref{thm:7} and $U_{k,\nu}$ are defined as in \eqref{eqn:control}. It follows from Fatou's Lemma that 
\begin{align*}
 \liminf_{\ell\to\iy}\E\int_0^T   k_1(V_{N_\ell}(s))ds
 \geq \E^*\int_0^T\int_\Lambda   k_1(V(s))ds.
\end{align*}
Another application of Fatou's Lemma shows 
\begin{align*}
\liminf_{\ell\to\iy}\E\int_0^T\int_\Lambda   k_2(u)m_s^{N_\ell}(du)ds
&\geq \E^*\int_0^T\int_\Lambda   k_2(u)m_s(du)ds\\
&\geq \E^*\int_0^T   k_2(U(s))ds
\end{align*}
where the second inequality follows on using Jensen's inequality, the relation \eqref{eqn:control}, and the assumed convexity of $k_2$. Thus
\begin{align*}
\liminf_{\ell\to\iy}J_{N_\ell}(U^{N_\ell},v_N)
&= \liminf_{\ell\to\iy}\E\int_0^T   (k_1(V_{N_\ell}(s))+k_2(\sqrt{N_\ell}U^{N_\ell}(s)))ds\\
&\geq \E\int_0^T   (k_1(V(s))+k_2(U(s)))ds\\
&\geq R(v_0),
\end{align*}
where the last inequality follows on noting that $U=(U_{k,\nu})_{k\in\mathbf{K},\nu\in\Del^k}\in\cla(\Xi)$ where $\Xi=(\Om^*,\clf^*,\PP^*,\{\clg_t\})$. This completes the proof of the theorem.
\end{proof}

\section{Feedback Controls}\label{Feedback}
In this section we will introduce feedback controls, $U^N_g\in\cla_N$ and $U_g\in\cla(\Xi)$, associated with a measurable map $g:[0,T]\times\RR^d\to\Lambda$ and prove that whenever $g$ is continuous and $v_N\to v_0$, we have, under suitable conditions, 
\begin{align}\label{eqn:costconv}
J_N(U^N_g,v_N)\to J(U_g,v_0).
\end{align}
In Section \ref{NetFeed} we introduce feedback controls for the $N$-th system, whereas in Section \ref{DiffFeed} we define feedback controls for the limit diffusion. For the latter case we argue, using the non degeneracy of $\alpha(t)$ (under Condition \ref{con:8}), that there is a unique weak solution of the corresponding stochastic differential equation. Finally, in Section \ref{FeedConv} we prove the convergence in \eqref{eqn:costconv} when $g$ is a continuous map. 

\subsection{Feedback Control in the $N$-th System}\label{NetFeed}
Given a measurable function $g:[0,T]\times\RR^d\to \Lambda$, define for all $k\in\mathbf{K},\nu\in\Del^k$, functions $\Gam_N^{k,g}(\cdot,\nu):\cls_N\times[0,T]\to\RR_+$ as 
\begin{align}\label{eqn:gamgdef}
\Gam_N^{k,g}(r,s,\nu)
\doteq \Gam_N^k\left(r,\frac{1}{\sqrt{N}}g(s,\sqrt{N}(r-\mu(s)),\nu\right).
\end{align}
As with $u\in\Lambda$, $g$ can be indexed by $k\in\mathbf{K}$ and $\nu\in\Del^k$ with the corresponding entry denoted as $g_{k,\nu}$. Define $\mu_N^g$ through the right side of \eqref{eqn:stpro} by replacing $U^N(s)$ with
\begin{align*}
U^N_g(s) \doteq \frac{1}{\sqrt{N}}g(s,\sqrt{N}(\mu^g_N(s)-\mu(s))).
\end{align*} 
Then it can be checked that $U^N_g\in\cla_N$ and $\mu^g_N$ is a time inhomogeneous Markov process with generator
\begin{align}\label{eqn:congen}
L^N_gf(s,r) \doteq\sum_{k=1}^K\sum_{\nu\in\Del^k}\Gam_N^{k,g}(r,s,\nu)\left[f\left(s,r+\frac{1}{N}e_\nu\right)-f(s,r)\right]
\end{align} 
for $s\in[0,T],\ r\in\cls_N,\ f:[0,T]\times\cls_N\to\RR$.

\subsection{Diffusion Feedback Control}\label{DiffFeed}
In this section we introduce feedback controls for the limit diffusion model. Fix $v_0\in\VV_{d-1}$.
\begin{definition}
Let $g:[0,T]\times\RR^d\to\Lambda$ be a measurable map. We say that the equation
\begin{align}\label{eqn:4}
\left\{
\begin{array}{rl}
dV(t)&=\eta(t,g(t,V(t)))dt+\beta(t)V(t)dt+\sigma(t)dW(t)\\
V(0)&=v_0
\end{array}
\right.
\end{align}
admits a weak solution if there exists a filtered probability space $(\Om,\clf,\PP,\{\clf_t\})$ on which is given an $\{\clf_t\}$-Wiener process $W$ and an $\clf_t$-adapted continuous process $V$ such that for all $0\leq t\leq T$,
\begin{align*}
V(t)=v_0+\int_0^t\eta(s,g(s,V(s)))ds+\int_0^t\beta(s)V(s)ds+\int_0^t\sigma(s)dW(s)
\end{align*}
almost surely. We say that \eqref{eqn:4} admits a unique weak solution if whenever there are two sets of such spaces and processes denoted as $(\Om^i,\clf^i,\PP^i,\{\clf_t^i\},(W^i,V^i)),\ i=1,2$ then the probability law of $V^1$ is the same as that of $V^2$.
\end{definition}

Given a weak solution $V$ associated with the system $\Xi=(\Om,\clf,\PP,\{\clf_t\},\{W_t\})$ define $U_g\doteq g(\cdot,V(\cdot))\in\cla(\Xi)$. We refer to this control as the feedback control (for the limit diffusion) associated with the map $g$.
\begin{theorem}\label{thm:2}
Under Condition \ref{con:8} there is a unique weak solution of \eqref{eqn:4}.
\end{theorem}
\begin{proof}
Suppose $V$ is a weak solution of \eqref{eqn:4} on some system $\Xi=(\Om,\clf,\PP,\{\clf_t\},\{W_t\})$. Recall the definition of $\hat{V},\ \hat{\eta},$ and $\hat{\beta}$  from Section \ref{LBnd} (cf. \eqref{eqn:hatdef1}, \eqref{eqn:hatdef2}, \eqref{eqn:hatdef3}). Let $Q' W\doteq\left(\begin{smallmatrix}B\\ W^*\end{smallmatrix}\right)$ and note that $B$ and $W^*$ are independent standard $(d-1)$ and 1 dimensional Brownian motions, respectively. Define $\hat{g}:[0,T]\times\RR^{d-1}\to\Lambda$ as $\hat{g}(t,v)=g(t,Q\left(\begin{smallmatrix}v\\ 0\end{smallmatrix}\right))$ and let $\left(\begin{smallmatrix}\hat{v}_0\\ 0\end{smallmatrix}\right)=Q'v_0$. Note that $\hat{V}$ is a solution of the $(d-1)$-dimensional SDE
\begin{align}\label{eqn:5}
\hat{V}(t)=\hat{v}_0+\int_0^t\hat{\eta}(s,\hat{g}(s,\hat{V}(s)))ds+\int_0^t\hat{\beta}(s)\hat{V}(s)ds+\int_0^t\alpha^{1/2}(s)dB(s).
\end{align}
On the other hand if $\hat{V}$ is a solution of the SDE \eqref{eqn:5} on some filtered probability space $(\Om,\clf,\PP,\{\clf_t\})$, where $B$ is a $(d-1)$-dimensional $\{\clf_t\}$ Brownian Motion, then as argued at the end of Theorem \ref{thm:7}, by a suitable augmentation of the space with a one-dimensional Brownian motion $B_d$, $Q\left[\begin{smallmatrix}\hat{V} \\ 0\end{smallmatrix}\right]$ is a solution of the SDE \eqref{eqn:4}, with Brownian motion $W=Q\left[\begin{smallmatrix}B \\ B_d\end{smallmatrix}\right]$. Since from \eqref{eqn:26} $\sup_{v\in\RR^d}\int_0^T\|\alpha(s)\|^{-1}\|\hat{\eta}(s,\hat{g}(s,v))\|^2ds<\iy$,  a standard argument using Girsanov's theorem shows that \eqref{eqn:5} has a unique weak solution. From the one-to-one correspondence between solutions of \eqref{eqn:5} and \eqref{eqn:4} noted above it now follows that there is a unique weak solution for \eqref{eqn:4}.
\end{proof}

Recall the generator $\cll_g$ in \eqref{eqn:ggen} associated with a measurable map
 $g:[0,T]\times \RR^d \to \Lambda$.
\begin{definition}
Given $v_0\in\VV_{d-1}$, a $d$-dimensional stochastic process $V$ on some filtered probability space $(\Om,\clf,\PP,\{\clf_t\})$  will be called a solution to the martingale problem associated with $(\cll_g,v_0)$ if
\begin{align*}
\phi(V(t))-\phi(v_0)-\int_0^t\cll_g\phi(s,V(s))ds
\end{align*}
is a martingale for all $\phi\in\CC_c^\iy(\RR^d)$ and $V(0)=v_0$ almost surely.
\end{definition}
The first part of the following result is standard (cf. \cite{stroock2007multidimensional}) whereas the second part is immediate from Theorem \ref{thm:2}.
\begin{theorem}\label{thm:martprob}
A process $V$ is a weak solution of the SDE \eqref{eqn:4} if and only if it is the solution to the martingale problem for $(\cll_g,v_0)$. In particular, under Condition \ref{con:8}, there is a unique solution
to the martingale problem for $(\cll_g,v_0)$.
\end{theorem}

\subsection{Convergence Under Continuous Feedback Controls}\label{FeedConv}
Let $g:[0,T]\times\RR^d\to\Lambda$ be a continuous function and $V^g$ be the unique solution to \eqref{eqn:4} given on some system $\Xi=(\Om,\clf,\PP,\{\clf_t\},\{W_t\})$. Define 
\begin{align}\label{eqn:veeNGdef}
V_N^g(t)=\sqrt{N}(\mu_N^g(t)-\mu(t)).
\end{align}
Recall that $U_g(t)=g(t,V^g(t))\in\cla(\Xi)$ and $U_g^N(t)=\frac{1}{\sqrt{N}}g(t,V^g_N(t))\in\cla_N$ are the controls associated with $g$ for the limit diffusion and pre-limit system, respectively. In this section we will show that $V^g_N$ converges in distribution to $V^g$, in $\DD([0,T]:\RR^d)$ and that $J_N(U_g^N,v_N)$ converges to $J(U_g,v_0)$. Namely we prove the following result.
\begin{theorem}\label{thm:3}
Suppose Conditions \ref{con:6}, \ref{con:7}, and \ref{con:8} hold, and let $v_N,v_0$ be as in Theorem \ref{thm:main}, where $x_N=\mu_N^g(0)$. Then as $N\to\iy$:
\begin{enumerate}
\item[(i)] $V_N^g$ converges in distribution, in $\DD([0,T]:\RR^d)$, to $V^g$ where $V^g$ is the unique solution to the martingale problem for $(\cll_g,v_0)$.
\item[(ii)] $J_N(U^g_N,v_N)\to J(U^g,v_0)$.
\end{enumerate}
\end{theorem}
\begin{proof}
First consider $(i)$. From Proposition \ref{prop:3} we have that $\{V_N^g\}$ is $\CC-$tight in $\DD([0,T]:\RR^d)$. 
Since $g$ is continuous, the operator $\cll_g$ defined in \eqref{eqn:ggen} maps $\CC_c^\iy(\RR^d)$ to $\CC_b([0,T]\times\RR^d)$. In view of this, the tightness of $\{V_N^g\}$, the uniqueness established in Theorem \ref{thm:martprob}, and Theorem 3.3.1 of \cite{joffe1986weak}, it suffices to show that for all $\phi\in \CC_c^\iy(\RR^d)$ 
\begin{align}\label{eqn:opconverge}
\lim_{N\to\iy}\int_0^T\E^N|\cll^N_g(\phi,s,V_N^g(s))-\cll_g\phi(s,V_N^g(s))|ds=0
\end{align}
where $\cll_g$ is an in \eqref{eqn:ggen} and $\cll^N_g$ is defined by the right side of \eqref{eqn:10}, replacing $u$ with $\frac{1}{\sqrt{N}}g(s,\sqrt{N}(s-\mu(s)))$, namely
\begin{align*}
\mathcal{L}_g^N(\phi,s,y)
&\doteq \sum_{k\in\mathbf{K}}\sum_{\nu\in\Del^k}\Gam_N^{k,g}\left(\gamma_N(s,y),s,\nu\right)\left[\phi\left(y+\frac{1}{\sqrt{N}}e_\nu\right)-\phi(y)\right]-\sqrt{N}F(\mu(s))\nabla\phi(y)
\end{align*}
for $\phi\in\CC_c^\iy(\RR^d),\ s\in[0,T],\ y\in\RR^d$ where $\Gam_N^{k,g}$ is as in \eqref{eqn:gamgdef}
(definition of $\Gam_N^{k,g}$  is extended to all $r\in \RR^d$ on setting $\Gam_N^{k,g}(r,s,\nu)=0$ if $r\not \in \cls_N$). We note that Theorem 3.3.1 of \cite{joffe1986weak} considers the setting of time-homogeneous diffusions, however the proof carries over to the setting of time-inhomogeneous generators considered here with minor modifications.

We now fix $\phi\in\CC_c^\iy(\RR^d)$ and for all $N\in\NN,\ k\in\mathbf{K},\ \nu\in\Del^k$ define functions $\vph_{k,\nu,1}^N:\RR^d\to \RR_+$, $\vph_{k,\nu,2}^N:[0,T]\times\RR^d\to \RR_+$, and $A_j^N:[0,T]\times\RR^d\to\RR_+$ for $j=1,2,3$, as
\begin{align*}
\vph_{k,\nu,1}^N(y)&\doteq\left|\phi\left(y+\frac{1}{\sqrt{N}}e_\nu\right)-\phi(y)-\frac{1}{\sqrt{N}}e_\nu'\nabla\phi(y)-\frac{1}{2N}e_\nu'D^2\phi(y)e_\nu\right|,\\
\vph_{k,\nu,2}^N(s,y)&\doteq\left|\beta_k^N(y,\mu(s),g(s,y),\nu)\right|,
\end{align*}
and
\begin{align*}
A_1^N(s,y)
&\doteq\sum_{k\in\mathbf{K}}\sum_{\nu\in\Del^k}\Gam_N^{k,g}\left(\gamma_N(s,y),s,\nu\right)\vph_{k,\nu,1}^N(y),\\
A_2^N(s,y)&\doteq \sum_{k\in\mathbf{K}}\sum_{\nu\in\Del^k}\vph_{k,\nu,2}^N(s,y)|e_\nu'\nabla\phi(y)|,\\
A_3^N(s,y)
&\doteq \frac{1}{2}\sum_{k\in\mathbf{K}}\sum_{\nu\in\Del^k}\left|\frac{1}{N}\Gam_N^{k,g}\left(\gamma_N(s,y),s,\nu\right)-\Gam^{k}(\mu(s),\nu)\right|\left|e_\nu'D^2\phi(y)e_\nu\right|
\end{align*}
for $s\in [0,T]$ and $y\in\RR^d$. Note that 
\begin{align*}
\Tr(a(t)D^2\phi)=\sum_{k\in\mathbf{K}}\sum_{\nu\in\Del^k}\Gam^k(\mu(t),\nu)e_\nu'D^2\phi(y)e_\nu.
\end{align*} 
Adding and subtracting 
\begin{align*}
\frac{1}{\sqrt{N}}\sum_{k\in\mathbf{K}}\sum_{\nu\in\Del^k}\Gam_N^{k,g}(\gamma_N(s,y),s,\nu)e_\nu'\nabla\phi(y)  \mbox{ and }
\frac{1}{2N}\sum_{k\in\mathbf{K}}\sum_{\nu\in\Del^k}\Gam_N^{k,g}(\gamma_N(s,y),s,\nu)e_\nu'D^2\phi(y)e_\nu
\end{align*}
from $\cll^N_g(\phi,s,y)-\cll_g\phi(s,y)$, the triangle inequality yields 
\begin{align*}
|\cll^N_g(\phi,s,V_N^g(s))-\cll_g\phi(s,V_N^g(s))|\leq A_1^N(s,V_N^g(s))+A_2^N(s,V_N^g(s))+A_3^N(s,V_N^g(s)).
\end{align*}

We now consider the three terms on the right side separately. First consider $A_1^N(s,V_N^g(s))$. It follows from Taylor's theorem and the fact that all derivatives of $\phi$ are uniformly bounded that there exists $\kappa_1\in(0,\iy)$ such that,
\begin{align*}
\vph_{k,\nu,1}^N(V_N^g(s))\leq \frac{1}{6}\max_{\|\alpha\|=3}\sup_{x\in\RR^d}\left\|D^\alpha\phi\left(x\right)\right\|\times \left\|\frac{e_\nu}{\sqrt{N}}\right\|^3\leq \frac{\kappa_1}{N^{3/2}},
\end{align*}
where the outside maximum is taken over all mixed derivatives of order 3. Then, since $\frac{1}{\sqrt{N}}V_N^g(s)+\mu(s)\in\cls_N$, \eqref{eqn:1} implies
\begin{align*}
A_1^N(s,V_N^g(s))\leq \sum_{k=\mathbf{K}}\sum_{\nu\in\Del^k}\Gam_N^{k,g}\left(\gamma_n(s,V_N^g(s)),s,\nu\right)\frac{\kappa_1}{N^{3/2}}
\leq \frac{\kappa_2}{\sqrt{N}},
\end{align*}
for all $s\in [0,T]$ and some $\kappa_2\in(0,\iy)$. It follows that 
\begin{align*}
\int_0^T\E^N|A_1^N(s,V_N^g(s))|ds \to 0\text{ as }N\to\iy.
\end{align*}

Now consider $A_2^N(s,V_N^g(s))$. From Condition \ref{con:7} it follows that for $\kappa_3>0$, $\eps>0$,
\begin{align*}
\PP^N\left[\sup_{0\leq s\leq T}\|V_N^g(s)\|\leq \kappa_3,\ \vph_{k,\nu,2}^N(s,V_N^g(s))>\eps\right]\to 0 \text{ as }  N\to \iy.
\end{align*}
Also the $\CC$-tightness of $\{V_N^g\}$ implies that
\begin{align*}
\sup_N\PP^N\left[\sup_{0\leq s\leq T}\|V_N^g(s)\|> \kappa_3\right]\to 0 \text{ as }  \kappa_3\to \iy.
\end{align*}
Combining these two observations we see that 
\begin{align}\label{eqn:sectermconv}
\vph_{k,\nu,2}^N(s,V_N^g(s)) \to 0 \text{ in probability as } N\to\iy \text{ for all }s\in [0,T].
\end{align}
Next, from Conditions \ref{con:6}, \ref{con:7}, and noting that $h_1,h_2$ are bounded functions, we see that there is a $\kappa_4\in(0,\iy)$ such that for all $k\in\mathbf{K},\ \nu\in\Del^k,\ N\geq 1$ and $s\geq 0$
\begin{align*}
\vph_{k,\nu,2}^N(s,V_N^g(s))\leq\kappa_4(1+\|V_N^g(s)\|)\text{ a.s.}
\end{align*}
From Proposition \ref{prop:3}, 
\begin{align}\label{eqn:momentbound}
\sup_N\E^N\sup_{t\leq T}\|V_N^g(t)\|^2<\iy.
\end{align}
Thus $\{\vph_{k,\nu,2}^N(s,V_N^g(s))\}$ is uniformly integrable over $[0,T]\times\Om$  and so combining this with \eqref{eqn:sectermconv}, we have
\begin{align*}
\int_0^T\E^N|\vph^{N}_{k,\nu,2}(s,V_N^g(s))|ds\to 0\text{ as }N\to\iy .
\end{align*}
 Recalling the definition of $A^N_2$, it follows from the fact that all derivatives of $\phi$ are uniformly bounded that there exists $\kappa_5\in(0,\iy)$ such that
\begin{align*}
\int_0^T\E^N|A^N_2(s,V_N^g(s))|ds
\leq \kappa_5\sum_{k\in\mathbf{K}}\sum_{\nu\in\Del^k}\int_0^T\E^N|\vph^N_{k,\nu,2}(s,V^g_N(s))|ds\to 0\text{ as }N\to\iy.
\end{align*}

Finally, consider $A_3^N(s,V_N^g(s))$. It follows from Condition \ref{con:6} and the boundedness of the derivatives of $\phi$ that there exists a $\kappa_6\in(0,\iy)$ such that,
\begin{align*}
A_3^N(s,V_N^g(s))\leq  \kappa_6\sum_{k=1}^K\sum_{\nu\in\Del^k}\left|\frac{1}{N}\Gam_N^{k,g}(\gamma_N(s,V_N^g(s)),s,\nu)-\Gam^k(\mu(s),\nu)\right|\leq\frac{\kappa_6 C_1}{\sqrt{N}}(1+\|V_N^g(s)\|).
\end{align*}
Using the moment bound in \eqref{eqn:momentbound} once more, we have that 
\begin{align*}
\int_0^T\E|A_3^N(s,V_N^g(s))|ds\to 0.
\end{align*}
This proves \eqref{eqn:opconverge} and thus completes the proof of part $(i)$.

Now consider $(ii)$. By a similar argument as in Theorem \ref{thm:7}
\begin{align*}
V_N^g(t)= v_N+\int_0^tb_N^{U_g^N(s)}(s,V_N^g(s))ds+M^N(t)\text{ for all }N\geq 1
\end{align*}
where $M^N(t)$ is the local martingale in \eqref{eqn:27}, with $M^N_{k,\nu}$ as in \eqref{eqn:mart} with $U^N$ replaced by $U^N_g$. Recall $p$ and $C_{k_1}$ introduced below \eqref{eqn:costprelim}. By a similar estimate as in \eqref{eqn:28} there exists $\kappa_7\in(0,\iy)$ such that 
\begin{align}\label{eqn:Mbound}
\begin{split}
&\sup_{N\in\NN}\E\sup_{0\leq t\leq T}\|M_i^N\|^{2p}
\leq\sup_{N\in\NN} \kappa_7N^{p}\sum_{k\in\mathbf{K}}\sum_{\nu\in\Del^k}\E[M^N_{k,\nu}]_T^{p}\\
&\qquad = \sup_{N\in\NN}\kappa_7\sum_{k\in\mathbf{K}}\sum_{\nu\in\Del^k}\E\left(\frac{1}{N}\mathcal{N}_{k,\nu}\left(\int_0^T\Gam^k_N(\mu_N(s),U^N(s),\nu)ds\right)\right)^{p}\\
&\qquad \leq\sup_{N\in\NN}\kappa_7\sum_{k\in\mathbf{K}}\sum_{\nu\in\Del^k}\E\left(\frac{1}{N}\mathcal{N}_{k,\nu}(NTC_2)\right)^{p}
<\iy
\end{split}
\end{align}
where $C_2$ is as in \eqref{eqn:1}. Also, from Lemma \ref{H1}
\begin{align}\label{eqn:bbound}
\|b_N^{U_g^N(s)}(s,V_N^g(s))\|^{2p}\leq \kappa_7(1+\|V_N^g(s)\|^{2p}).
\end{align}
 Combining these two inequalities  implies there exists a $\kappa_8\in(0,\iy)$ such that
\begin{align*}
\E\sup_{0\leq s\leq t}\|V_N^g(s)\|^{2p}\leq \kappa_8\left(1+\int_0^tE\sup_{0\leq u\leq s}\|V_N^g(u)\|^{2p}ds\right)\text{ for all } 0\leq t\leq T.
\end{align*}
Gronwall's inequality then yields, 
\begin{align*}
\sup_{N\in\NN}\E\sup_{0\leq t\leq T}\|V_N^g(t)\|^{2p}\leq \sup_{N\in\NN}\kappa_8e^{\kappa_8 T}<\iy
\end{align*}
and thus $\{\sup_{t\leq T}\|V_N^g(t)\|^{p}\}$ is uniformly integrable. Recalling the definition of $J_N$ in \eqref{eqn:costprelim}, it follows from this uniform integrability, part $(i)$ of the theorem, the compactness of $\Lambda$, and growth condition on $k_1$ (see below \eqref{eqn:costprelim}) that
\begin{align*}
&\E\int_0^T (k_1(V_N^g(t))+k_2(\sqrt{N}U^N_g(t)))dt\to \E\int_0^T (k_1(V^g(t))+k_2(U_g(t)))dt,
\end{align*}
upon noting that $\sqrt{N}U^N_g(t)=g(t,V^g_N(t)),\ U_g(t)=g(t,V^g(t))$, and $g$ is continuous. Thus we have shown $J_N(U_g^N,v_N)\to J(U_g,v_0)$ which completes the proof of $(ii)$.
\end{proof}

\section{Near Optimal Continuous Feedback Control}\label{NearOptConCont}
In this section we give the final ingredient in the proof of Theorem \ref{thm:main}, namely Theorem \ref{thm:5}. This result says that for every $v_0\in\VV_{d-1}$ and $\eps>0$ there is a continuous $g_\eps:[0,T]\times\RR^d\to\Lambda$ such that $U_{g_\eps}$ is an $\eps$-optimal control for the diffusion control problem, i.e. $J(U_{g_\eps},v_0)\leq R(v_0)+\eps$. Recall from Section \ref{MainRes} that this result combined with Theorems \ref{thm:6} and \ref{thm:3} proved earlier will complete the proof of Theorem \ref{thm:main}. We begin with a result that says that for every $v_0\in\VV_{d-1}$, the infimum of the cost $J(\cdot,v_0)$ over all controls is the same as that over all feedback controls. The proof is similar to Theorem 4.2 in \cite{borkar1989optimal} which considers a time homogeneous setting, and so we only provide a sketch.

Recall that for every measurable $g:[0,T]\times\RR^d\to\Lambda$ there is a (feedback) control $U_g\in\cla(\Xi)$ on some system $\Xi$. Denote the family of all such feedback controls as $\cla_{fb}$. (This class depends on the initial condition $v_0$ in \eqref{eqn:4} but we suppress this in the notation).
Throughout this section we will assume that Conditions \ref{con:6} -- \ref{con:8} hold.
\begin{theorem}\label{thm:4}
Fix $v_0\in\VV_{d-1}$. Then 
\begin{align*}
R(v_0)=\inf_{U\in\cla_{fb}}J(U,v_0).
\end{align*}
\end{theorem}
\begin{proof}
Suppose $U\in\cla(\Xi)$ is an admissible control  on a system $\Xi=(\Om,\clf,\PP,\{\clf_t\},\{W_t\})$. 
As in Section \ref{LBnd} (cf. \eqref{eqn:relcon}) we denote the corresponding relaxed control by $m$.
Let $V(\cdot)$ be the corresponding unique pathwise solution to \eqref{eqn:6}. It suffices to show that there exists an admissible feedback control $U^*$ such that $J(U^*,v_0)\leq J(U,v_0)$. Define the probability measure $\nu_{v_0}\in\clp([0,T]\times\VV_{d-1}\times\Lambda)$ as
\begin{align*}
\int_{[0,T]\times\VV_{d-1}\times\Lambda}f(t,x,u)d\nu_{v_0}(t,x,u)
&= \frac{1}{T}\E\left[\int_0^T \int_\Lambda f(t,V(t),u)m_t(du)\ dt\right]
\end{align*}
for all $f\in\CC_b([0,T]\times\VV_{d-1}\times\Lambda)$. Disintegrate $\nu_{v_0}$ as
\begin{align*}
\nu_{v_0}(dt\ dx\ du)=\beta_{v_0}(dt,dx)\pi(t,x)(du)
\end{align*}
where $\beta_{v_0}\in\clp([0,T]\times\VV_{d-1})$ is the marginal distribution of $\nu_{v_0}$ on the first two coordinates
 and $\pi:[0,T]\times\VV_{d-1}\to\clp(\Lambda)$ is the corresponding regular conditional law. Define $g^*:[0,T]\times\RR^d\to\Lambda$ as $g^*(t,x)=\int_\Lambda u\pi(t,\Pi_{\VV_{d-1}}(x))(du)$ where $\Pi_{\VV_{d-1}}:\RR^d\to\VV_{d-1}$ is the projection of $\RR^d$ onto $\VV_{d-1}$. Let $U_{g^*}$ be the feedback control associated with the map $g^*$ given on some system $\Xi^*$ and let $V^*$ be the corresponding state process given as the solution of \eqref{eqn:4} with $g$ replaced by $g^*$. 
Let for $t\in[0,T],\ \pi_t\doteq \pi(t,V^*(t))$. For $(t,z)\in[0,T]\times\VV_{d-1}$, $r \in (0, \infty)$ and $\bar{k}_r(v,u) \doteq k_1(v)\wedge r+k_2(u)$ define
\begin{align*}
\phi_{r}(t,z) &= \E^*\left[\int_t^T \int_\Lambda \bar{k}_r(V^*(s),u)\pi_s(du)ds\Big |
V^*(t)=z\right],\\
Y_{r}(t) &= \int_0^t   \int_\Lambda \bar{k}_r(V(s),u)m_s(du)ds+ \phi_{r}(t,V(t)).
\end{align*}
It follows using the equivalent description of a weak solution of \eqref{eqn:4} in terms of a $(d-1)$-dimensional SDE with uniformly non-degenerate diffusion coefficient as in the proof of Theorem \ref{thm:2} and classical PDE results (cf. Section III.4.2 of \cite{bensoussan2011stochastic}) that $\phi_{r}$ solves  the equation
\begin{align}\label{eqn:pde}
\int_\Lambda \bar{k}_r(x,u)\pi(t,x)(du)+\frac{\partial}{\partial t}\phi_{r}(t,x)+(\cll_{g^*}\phi_{r})(t,x)=0
\end{align}
where $\cll_{g^*}$ is the generator for $V^*$ given by the right side of \eqref{eqn:ggen} with $g$ replaced by $g^*$.
From It\^{o}-Krylov formula(cf. \cite{Krylovbook}) we have
\begin{align}\label{eqn:ydiff}
\begin{split}
\E[Y_{r}(t)]-\E[Y_{r}(0)]
& = \E\int_0^t  \Big(\int_\Lambda\bar{k}_r(V(s),u)m_s(du)+\frac{\partial}{\partial t}\phi_{r}(s,V(s))\\
&\quad\quad +(\hat \cll_{U(s)}\phi_{r})(s,V(s))\Big)ds.
\end{split}
\end{align}
where for $u\in\Lambda$, $\hat \cll_u$ is the ``controlled generator'' defined as
\begin{align*}
\hat \cll_u\phi_{r}(t,x)
&= \nabla_x\phi_{r}(t,x)(\eta(t,u)+\beta(t)x)+\frac{1}{2}\Tr(\sigma(t)D^2\phi_{r}(t,x)\sigma'(t)).
\end{align*}
By the definition of $\pi$, and since $u\mapsto \hat \cll_u \phi_{r}(t,x)$ is linear we see that
$$\int_{\Lambda} (\hat \cll_u\phi_{r} )(s,x) \pi(s,x)(du) = (\cll_{g^*}\phi)(s,x), \; (s,x) \in [0, T] \times \VV_{d-1}.$$
From this it follows that
\begin{align*}
&\E\int_0^t   \left(\int_\Lambda \bar{k}_r(V(s),u)m_s(du)+(\hat \cll_{U(s)}\phi_{r})(s,V(s))\right)ds\\
&\qquad=\E\int_0^t   \left(\int_\Lambda \bar{k}_r(V(s),u)\pi(s,V(s))(du)+(\cll_{g^*}\phi_{r})(s,V(s))\right)ds.
\end{align*}
Thus \eqref{eqn:pde} implies that the right hand side of \eqref{eqn:ydiff} is 0 and thus $\E[Y_{r}(t)]=\E[Y_{r}(0)]=\phi_{r}(0,v_0)$ for all $t\in [0,T]$. From the convexity of $k_2$ we see that
\begin{align*}
\phi_{r}(0,v_0)
&= \E^*\left[\int_0^{T}  \int_\Lambda\bar{k}_r(V^*(s),u)\pi_s(du) ds\right]\\
&\geq \E^*\left[\int_0^{T}   \bar{k}_r(V^*(s),g^*(s,V^*(s)) ds\right]\\
&\doteq J_{r}(U_{g^*},v_0).
\end{align*}
Using monotone convergence theorem it now follows that 
\begin{align*}
	J(U,v_0)&=\lim_{r\to\infty}\E[Y_{r}(T)]
= \lim_{r\to\infty}  \E Y_{r}(0)\\
&\quad \quad  = 
\lim_{r\to\infty} \phi_{r}(0,v_0)\geq \lim_{r\to\infty} J_{r}(U_{g^*},v_0) =
J(U_{g^*},v_0).
\end{align*}
 The result follows.
\end{proof}
We will next show in Theorem \ref{thm:5} below that the above theorem can be strengthened in that the class $\cla_{fb}$ can be replaced by the smaller class $\cla_{fb}^c$ of all \emph{continuous} feedback controls, i.e. feedback controls for which that corresponding map $g$ is continuous. 
Recall the orthogonal matrix $Q$ defined in Section \ref{DiffCont}. 
Fix $v_0\in\VV_{d-1}$ and let $g^*:[0,T]\times\RR^d\to\Lambda$ be a measurable map. Let $U_{g^*}$ be the corresponding feedback control given on some system $\Xi=(\Om,\clf,\PP,\{\clf_t\},\{W_t\})$ and let $V^*$ be the solution of \eqref{eqn:4} with $g$ replaced by $g^*$ on the right side. Define the $(d-1)$-dimensional process $\hat{V}^*$ such that $V^*=Q\left(\begin{smallmatrix}\hat{V}^*\\ 0\end{smallmatrix}\right)$ and the map $\hat{g}^*:[0,T]\times\RR^{d-1}\to\Lambda$ as $\hat{g}^*(t,v)= g^*(t,Q\left(\begin{smallmatrix}v\\0\end{smallmatrix}\right))$ for $v\in\RR^{d-1}$. Then,
\begin{align}\label{eqn:22}
\hat{V}^*(t)
&= \hat{v}_0+\int_0^t\hat{\eta}(s,\hat{g}^*(s,\hat{V}^*(s)))ds+\int_0^t\hat{\beta}(s)V^*(s)ds+\int_0^t\alpha^{1/2}(s)d\hat{W}(s)
\end{align}
where $\hat{\eta},\ \hat{\beta}$, and $\alpha$ are as in \eqref{eqn:hatdef1}, \eqref{eqn:hatdef2}, and \eqref{eqn:alphadef}, respectively. In addition, $v_0=Q\left(\begin{smallmatrix}\hat{v}_0\\0\end{smallmatrix}\right)$ and $Q'W=\left(\begin{smallmatrix}\hat{W}\\B_d\end{smallmatrix}\right)$.
Define $\varrho\in\clp([0,T]\times\RR^{d-1})$ as 
\begin{align}\label{eqn:24}
\varrho(A)\doteq\bar{c}\int_Ae^{-\frac{(\|x\|^2+t^2)}{2}}dxdt
\end{align}
for $A\in\clb([0,T]\times\RR^{d-1})$ where $\bar{c}$ is a normalizing constant. We denote by $\bar{\clb}$ the Lebesgue $\sigma$-field on $[0,T]\times\RR^{d-1}$, namely the completion of $\clb([0,T]\times\RR^{d-1})$ with respect to the Lebesgue measure.

\begin{lemma}\label{lem:2}
For each $n\in\NN$ there exists a $\bar{\clb}$-measurable function $\hat{g}_n:[0,T]\times\RR^{d-1}\to\Lambda$ and compact sets $A_n\in\clb([0,T]\times\RR^{d-1})$ such that $\hat{g}_n$ is continuous and,
\begin{align}\label{eqn:20}
\{(s,v)\in[0,T]\times\RR^{d-1}:\hat{g}^*(s,v)\neq \hat{g}_n(s,v)\}\subset A_n^c\quad\text{ and }\quad \varrho(A_n^c)\leq\frac{1}{2^{n+1}}.
\end{align}
\end{lemma}
\begin{proof}
From Lusin's theorem (c.f. 2.24 of \cite{rudin1986real}) for each $n\in\NN$ there exists a continuous function $\hat{g}'_n:[0,T]\times\RR^{d-1}\to\RR^\mathbf{\ell}$ such that \eqref{eqn:20} is satisfied. Since $\Lambda$ is a closed convex set, there is a continuous map $\Pi_\Lambda:\RR^\mathbf{\ell}\to\Lambda$ such that $\Pi_\Lambda(u)=u$ for all $u\in\Lambda$. Define $\hat{g}_n:[0,T]\times\RR^{d-1}\to\Lambda$ as $\hat{g}_n(s,v)=\Pi_\Lambda(\hat{g}_n'(s,v))$. The result now follows on noting that
\begin{align*}
\{(s,v):\hat{g}_n(s,v)=\hat{g}^*(s,v)\}\supset\{(s,v):\hat{g}_n'(s,v)=\hat{g}^*(s,v)\}.
\end{align*}
\end{proof}

Let $\{v_n\}\subset\VV_{d-1}$ be such that $v_n\to v_0$ and let $\Xi^n=(\Om^n,\clf^n,\{\clf^n_t\},\PP^n,\{W^n\})$ be a system on which the process $V^n$ is the unique (weak) solution to
\begin{align}\label{eqn:21}
V^n(t)=v_n+\int_0^t\eta(s,g_n(s,V^n(s)))ds+\int_0^t\beta(s)V^n(s)ds+\int_0^t\sigma(s)dW^n(s)
\end{align}
where $g_n:[0,T]\times \VV_{d-1}\to\Lambda$ is the continuous function defined as $g_n(s,Q\left(\begin{smallmatrix}v\\0\end{smallmatrix}\right))=\hat{g}_n(s,v),\ v\in\RR^{d-1}$. 
We can extend $g_n$ continuously to $[0,T]\times \RR^{d}$ as before using the projection map
$\Pi_{\VV_{d-1}}$.
Defining $\hat{V}^n$ as $Q'V^n=\left(\begin{smallmatrix}\hat{V}^n\\0\end{smallmatrix}\right)$, we can write
\begin{align*}
\hat{V}^n(t)
&= \hat{v}_n+\int_0^t\hat{\eta}(s,\hat{g}_n(s,\hat{V}^n(s)))ds+\int_0^t\hat{\beta}(s)\hat{V}^n(s)ds+\int_0^t\alpha^{1/2}(s)d\hat{W}^n(s)
\end{align*}
where $Q'v_n=\left(\begin{smallmatrix}\hat{v}_n\\0\end{smallmatrix}\right)$ and $\hat{W}^n$ is a $(d-1)$-dimensional Brownian motion.
\begin{theorem}\label{thm:5}
Given $v_0\in\VV_{d-1}$, let $V^*$ be as introduced in  \eqref{eqn:22}. Let $v_n,\ g_n$ and $\{V^n\}$ be as introduced above. Then $V^n\Rightarrow V^*$ as a sequence of $\CC([0,T]:\RR^d)$-valued random variables.
\end{theorem}
\begin{proof}
It suffices to show that $\hat{V}^n\Rightarrow \hat{V}^*$. Let $G=\RR^{d-1}\times\Lambda$ and define $m^n\in\clm([0,T]\times G)$ as
\begin{align*}
m^n(A\times B\times C)\doteq\int_0^T 1_A(s)1_B(\hat{V}^n(s))1_C(\hat{g}_n(s,\hat{V}^n(s)))ds,
\end{align*}
where $A\in\clb([0,T]),\ B\in\clb(\RR^{d-1}),\ C\in\clb(\Lambda)$. Since $u \mapsto \hat{\eta}(s,u)$ is a linear function  and $\int_0^t\hat{g}_n(s,\hat{V}^n(s))ds=\int_0^tum^n(ds\ dv\ du)$, $\hat{V}^n(t)$ can be expressed as
\begin{align*}
\hat{V}^n(t)&= \hat{v}_n+\int_0^t\hat{\eta}(s,u)m^n(ds\ dv\ du)+\int_0^t\hat{\beta}(s)\hat{V}^n(s)ds+\int_0^t\alpha^{1/2}(s)d\hat{W}^n(s).
\end{align*}
We  can disintegrate $m^n$ as $m_t^n(dv\ du)dt$, where $m_t^n(dv\ du)=\del_{\hat{V}^n(t)}(dv)\del_{\hat{g}_n(t,\hat{V}^n(t))}(du)$ and $\del_x$ is the Dirac measure at the point $x$. From the boundedness of $\hat{\eta},\ \hat{\beta},$ and $\alpha^{1/2}$, we get by a standard application of Gronwall's inequality that for some $C\in(0,\iy)$
\begin{align}\label{eqn:momentbound2}
\E[\hat{V}^n(t)]\leq C(1+\hat{v}_n)e^{Ct},\text{ for all }n\in\NN,\ t\in[0,T].
\end{align}
Using this moment bound and a similar bound on the increments of $\hat{V}^n$ we have that $\{\hat{V}^n\}$ is a tight sequence of $\CC([0,T]:\RR^{d-1})$-valued random variables. Now the tightness of $\{m^n\}$ as a sequence of $\clm([0,T]\times G)$-valued random variables is immediate since the first marginal is the Lebesgue measure (i.e. $m^n([0,t)\times G)=t$ for all $t\in[0,T]$), $\{\hat{V}^n\}$ is tight, and $\Lambda$ is compact. Also, the tightness of $\{\hat{W}^n\}$ as a sequence of $\CC([0,T]:\RR^{d-1})$-valued random variables is immediate since $\hat{W}^n$ is a standard Brownian motion for each $n$. Therefore $\{\hat{V}^n,\hat{W}^n,m^n\}$ is a tight collection of $\CC([0,T]:\RR^{2(d-1)})\times\clm([0,T]\times G)$-valued random variables.

Suppose $\{\hat{V}^n,\hat{W}^n,m^n\}$ converges along a subsequence (also denoted $\{n\}$) to a process, $\{\hat{V},\hat{W},m\}$. Let $(\Om',\clf',\PP')$ be the probability space on which the limit processes are defined. Then $\hat{W}$ is a $\PP'$-Brownian motion and using the continuity of $\hat{\eta},\ \hat{\beta}$ and $\alpha^{1/2}$ we see that $(\hat{V},\hat{W},m)$ satisfy
\begin{align*}
\hat{V}(t) = \hat{v}_0+\int_0^t\hat{\eta}(s,u)dm(ds\ dv\ du)+\int_0^t\hat{\beta}(s)\hat{V}(s)ds+\int_0^t\alpha^{1/2}(s)d\hat{W}(s)
\end{align*}
$\PP'$-almost surely.

Define $\clf_t'=\sigma\{\hat{V}_s,\hat{W}_s,m([0,s]\times A):0\leq s\leq t, A \in \clb(G)\}$. It is easy to check that $\{\hat{W}_t\}$ is a $\{\clf_t'\}$-martingale. Indeed, let $k\in\NN$ and $\clh:(\RR^{2(d-1)}\times\RR)^k\to\RR$ be a bounded and continuous function. Define $\clz_t\doteq(\hat{V}_t,\hat{W}_t,m(t,f))$ and $\clz_t^n\doteq(\hat{V}_t^n,\hat{W}_t^n,m^n(t,f))$,
where $f \in \CC_b(G)$ and $\nu(t,f) = \int_0^t f(v,u) \nu(ds\ dv\ du)$ for $\nu = m, m^n$. Then for $s\leq t\leq T$ and $0\leq t_1\leq\ldots t_k\leq s$,
\begin{align*}
\E'\clh(\clz_{t_1},\ldots,\clz_{t_k})[\hat{W}_t-\hat{W}_s]
= \lim_{n\to\iy}\E^n\clh(\clz^n_{t_1},\ldots,\clz^n_{t_k})[\hat{W}^n_t-\hat{W}^n_s]
=0,
\end{align*} 
where the second equality uses the fact that $\hat{W}^n$ is a $\{\clf_t^n\}$-Brownian motion and $\clz^n_t$ is $\{\clf_t^n\}$-adapted. This proves that $(\hat{W}_t)$ is an $\{\clf_t'\}$-martingale.

Note that $m, m^n$ can be disintegrated as
$$m(ds\ dv\ du) = m_s(dv\ du) ds,\; m^n(ds\ dv\ du) = m_s^n(dv\ du) ds.$$
We will now argue that for all $t\in[0,T]$,
\begin{align}\label{eqn:37}
\int_0^t\int_Gum_s(dv\ du)ds=\int_0^t\hat{g}^*(s,\hat{V}(s))ds\quad\text{a.s. }\PP'.
\end{align}
Note that \eqref{eqn:37}, the linearity of $\hat{\eta}$ in $u$, together with the weak-uniqueness of solutions to \eqref{eqn:22} (which was established in Section \ref{Feedback}) completes the proof of the theorem. 

Note that for any $f\in\CC_b(\RR^{d-1})$ we have $\int_0^t\int_Gf(v)m_s^n(dv\ du)ds=\int_0^tf(\hat{V}^n(s))ds$. Since $(m^n,\hat{V}^n)\Rightarrow(m,\hat{V})$, we have for any such $f$
\begin{align*}
\int_0^t\int_Gf(v)m_s(dv\ du)ds=\int_0^tf(\hat{V}(s))ds\quad\text{ for all }t\in[0,T],\ \text{a.s. }\PP'.
\end{align*}
Denote by $\hat{m}_t^{i},\ i=1,2$ the marginal of $m_t$ on its $i$-th coordinate. Then the above display can be rewritten as
\begin{align*}
\int_0^t\int_{\RR^{d-1}}f(v)\hat{m}_s^1(dv)ds=\int_0^tf(\hat{V}(s))ds,\text{ for }t\in[0,T], \text{ a.s. }\PP',\text{ for every }f\in\CC_b(\RR^{d-1}:\RR).
\end{align*}
This shows that
\begin{align}\label{eqn:36}
\hat{m}_t^1(dv)=\del_{\hat{V}(t)}(dv),\ [\lambda\otimes\PP']\text{ a.e. }(t,w')
\end{align}
where $\lambda$ is the Lebesgue measure on $[0,T]$.

Recall the definition of $A_n$ from Lemma \ref{lem:2}  and $\varrho$ from \eqref{eqn:24}. Define $B_n\doteq\cap_{m=n}^\iy A_n$. Then
\begin{align*}
\varrho(B_n)\geq 1-\frac{1}{2^n}\quad\text{ for all }n\geq1
\end{align*}
and $\hat{g}^*(s,v)=\hat{g}_n(s,v)=\hat{g}_{n+1}(s,v)=\ldots$ for all $(s,v)\in B_n$. Since $\{\hat{v}_n\}$ is bounded we have from the moment bound in \eqref{eqn:momentbound2} that for every $\eps>0$, there is a compact $F\subset\RR^{d-1}$ such that
\begin{align}\label{eqn:29}
\sup_{n\in\NN}\sup_{0\leq t\leq T}\PP^n[\hat{V}^n(t)\in F^c]\leq\frac{\eps}{2}.
\end{align}
Note that this says in particular that $\{\hat{v}_n\}\subset F$. For $t\in[0,T]$ and $v\in\RR^{d-1}$, let $p(t,v,z)$ be the transition probability density of the Gaussian random variable $\hat{V}_0^v(t)$ given as the solution of the SDE
\begin{align*}
\hat{V}_0^{v}(t)=v+\int_0^t\hat{\beta}(s)\hat{V}^v_0 ds+\int_0^t\alpha^{1/2}(s)d\hat{W}(s).
\end{align*} 
It is easy to see that there exists a function $\Psi:[0,T]\to\RR_+$ and $\kappa\in(0,\iy)$ such that 
\begin{align}\label{eqn:30}
\sup_{v,z\in F}p(t,v,z)\leq\Psi(t),\ t\in[0,T],\text{ and }
\int_0^Te^{-\kappa/t}\Psi(t)dt<\iy.
\end{align}

Using the boundedness of $\hat{\eta}$ and $\alpha^{-1/2}$, Girsanov's theorem, and the Cauchy-Schwarz inequality we see that there exists a $\theta\in(0,\iy)$ such that for any bounded measurable $f:[0,T]\times\RR^{d-1}\to\RR$ and $t\in[0,T]$
\begin{align}\label{eqn:31}
\E^n\left|\int_0^tf(s,\hat{V}^n(s))ds\right|\leq \theta\left[\E'\left(\int_0^tf(s,\hat{V}^{v_n}_0(s))^2ds\right)\right]^{1/2}.
\end{align} 
Since $e^{-\kappa/s}\psi(s)1_F(v)dvds$ is a finite measure on $[0,T]\times\RR^{d-1}$ that is absolutely continuous with respect to $\varrho$, we have for any $\eps>0$ a $n_0\in\NN$ such that
\begin{align}\label{eqn:32}
\int_0^T\int_{\RR^{d-1}}1_{B_{n_0}^c}(s,v)e^{-\kappa/s}1_{F}(v)\Psi(s)dvds<\frac{\eps^2}{4\theta^2}.
\end{align}
Together with \eqref{eqn:30}, \eqref{eqn:32} implies
\begin{align}\label{eqn:33}
\E'\int_0^Te^{-\kappa/s}1_{B_{n_0}^c}(s,\hat{V}_0^{v}(s))1_{F}(\hat{V}_0^{v}(s))ds<\frac{\eps^2}{4\theta^2}
\end{align}
for all $v\in F$. From \eqref{eqn:29}, \eqref{eqn:31}, and \eqref{eqn:33} we have
\begin{align}\label{eqn:34}
\begin{split}
\E^n\int_0^Te^{-\kappa/2s}1_{B_{n_0}^c}(s,\hat{V}^n(s))ds
&< \E^n\int_0^T1_F(\hat{V}^n(s))e^{-\kappa/2s}1_{B_{n_0}^c}(s,\hat{V}^n(s))ds+\frac{\eps}{2}\\
&\leq\theta\left[\E'\left(\int_0^T1_F(\hat{V}^{v_n}_0(s))e^{-\kappa/s}1_{B_{n_0}^c}(s,\hat{V}^{v_n}_0(s))ds\right)\right]^{1/2}+\frac{\eps}{2}\\
&\leq\eps. 
\end{split}
\end{align}
Denote by $\hat{m}_t^{n,i}$ the marginal of $m_t^n$ on the $i$-th coordinate for $i=1,2$. Then, for any $n\geq n_0,\ t\in[0,T],\ f\in\CC(\Lambda)$, and $h\in\CC([0,T])$
\begin{align*}
\int_0^t\int_Ge^{-\kappa/2s}h(s)f(u)m_s^n(dv\ du)ds
&= \int_0^t\int_{\RR^{d-1}}e^{-\kappa/2s}h(s)f(\hat{g}_n(s,v))\hat{m}_s^{n,1}(dv)ds\\
&= \int_0^t\int_{\RR^{d-1}}1_{B_{n_0}}(s,v)e^{-\kappa/2s}h(s)f(\hat{g}_{n_0}(s,v))\hat{m}_s^{n,1}(dv)ds\\
&\qquad+\int_0^t\int_{\RR^{d-1}}1_{B_{n_0}^c}(s,v)e^{-\kappa/2s}h(s)f(\hat{g}_n(s,v))\hat{m}_s^{n,1}(dv)ds,
\end{align*}
where the second equality follows on noting that for $(s,v)\in B_{n_0}$, $\hat{g}_n(s,v)=\hat{g}_{n_0}(s,v)$ when $n\geq n_0$. Thus
\begin{align}\label{eqn:35}
\begin{split}
&\left|\int_0^t\int_Ge^{-\kappa/2s}h(s)f(u)m_s^n(dv\ du)ds-\int_0^t\int_{\RR^{d-1}}e^{-\kappa/2s}h(s)f(\hat{g}_{n_0}(s,v))\hat{m}_s^{n,1}(dv)ds\right|\\
&\qquad\leq 2\|f\|_\iy\|h\|_\iy\int_0^t\int_{\RR^{d-1}}1_{B_{n_0}^c}(s,v)e^{-\kappa/2s}\hat{m}_s^{n,1}(dv)ds.
\end{split}
\end{align}
 It follows from \eqref{eqn:34} that the expectation of \eqref{eqn:35} is bounded above by $2\|f\|_\iy\|h\|_\iy\eps$ and thus, letting $n\to\iy$
\begin{align*}
&\E'\left|\int_0^t\int_Ge^{-\kappa/2s}h(s)f(u)m_s(dv\ du)ds-\int_0^t\int_{\RR^{d-1}}e^{-\kappa/2s}h(s)f(\hat{g}_{n_0}(s,v))\hat{m}^1_s(dv)ds\right|\\
&\qquad\leq 2\|f\|_\iy\|h\|_\iy\eps.
\end{align*}
Therefore, since $\hat{g}_{n_0}(s,v)=\hat{g}^*(s,v)$ on $B_{n_0}$
\begin{align*}
&\E'\left|\int_0^t\int_Ge^{-\kappa/2s}h(s)f(u)m_s(dv\ du)ds-\int_0^t\int_{\RR^{d-1}}e^{-\kappa/2s}h(s)f(\hat{g}^*(s,v))\hat{m}^1_s(dv)ds\right|\\
&\qquad\leq 2\|f\|_\iy\|h\|_\iy\left[\eps+\E'\int_0^t\int_{\RR^{d-1}}1_{B^c_{n_0}}(s,v)e^{-\kappa/2s}\hat{m}^1_s(dv)ds\right].
\end{align*}
Since $B_{n_0}^c$ is open, it then follows from \eqref{eqn:34}
\begin{align*}
\E'\int_0^t\int_{\RR^{d-1}}1_{B^c_{n_0}}(s,v)e^{-\kappa/2s}\hat{m}^1_s(dv)ds
\leq \liminf_{n\to\iy}\E^n\int_0^t\int_{\RR^{d-1}}1_{B^c_{n_0}}(s,v)e^{-\kappa/2s}\hat{m}_s^{n,1}(dv)ds
\leq\eps.
\end{align*}
Letting $\eps\to0$ we have for all $t\in[0,T],\ h\in\CC([0,T]),\ f\in\CC(\Lambda)$ that
\begin{align*}
\int_0^t\int_Ge^{-\kappa/2s}h(s)f(u)m_s(dv\ du)ds=\int_0^t\int_{\RR^{d-1}}e^{-\kappa/2s}f(\hat{g}^*(s,v))\hat{m}^1_s(dv)ds\quad\text{a.e. }\PP'.
\end{align*}
Combined with \eqref{eqn:36} this implies that 
\begin{align*}
m_s(dv\ du)=\del_{\hat{V}(s)}(dv)\del_{\hat{g}^*(s,\hat{V}(s))}(du),[\lambda\times\PP']\text{ a.e. }(s,w').
\end{align*}
This proves \eqref{eqn:37} and, as argued previously, completes the proof of the theorem.
\end{proof}

\section{Example}\label{Ex1}

The following class of models is studied in \cite{antunes2008stochastic}. Consider a system consisting of $N$ identical servers (nodes) of capacity $C\in\NN$ and $K$ different classes of jobs each with its own capacity requirement $A_k\in\NN,\ k\in\{1,\ldots,K\}$. External jobs of type $k$ arrive at each server with rate $\lambda_k$. A job of type $k$ remains at a given node for an exponential holding time with mean $\gamma_k^{-1}$ before attempting to move to another randomly chosen node. If the server has available capacity it accepts the job, otherwise the job is rejected and exits the system. If not rejected first, a type $k$ job remains in the system for an exponential amount of time with mean $\tau_k^{-1}$ before leaving the system. We make the usual assumptions of mutual independence, in particular a.s. at most one job may arrive, switch nodes, or exit the system at a given time, but note that such an event may correspond to the change in state of multiple servers.

For the discussion below, for simplicity, we consider the case where there are only two classes of jobs. In the notation of the current paper, the state process $X_N(t)=\{X_N^1(t),\ldots,X_N^N(t)\}$ is the pure jump Markov process where $X_N^i(t)$ takes values in 
\begin{align*}
\XX=\{(j,i)\in\NN_0\times\NN_0:jA_1+iA_2\leq C\}.
\end{align*} 
Let, as before, $d= |\XX|,\ \cls=\clp(\XX)$, and $\cls_N=\clp(\XX)\cap\frac{1}{N}\NN^d$. The empirical measure process, $\mu_N(t)\in\cls_N$, is a $d$-dimensional pure jump Markov process where $\mu_N^{j,i}(t)=\frac{1}{N}\sum_{k=1}^NI_{\{X_N^k(t)\}}((j,i))$ represents the proportion of nodes with exactly $j$ and $i$ jobs of type 1 and 2, respectively. We suppose that $\mu_N(0)=x_N$ a.s. for some deterministic $x_N\in\cls_N$ such that $x_N\to x_0$ as $N\to\iy$ and $x_0^{j,i}>0$ for all $(j,i)\in\XX$. Also suppose that $v_N\doteq\sqrt{N}(x_N-x_0)\to v_0$ as $N\to\iy$. The rate function $\bar \Gamma_N^k$ associated with this system is described in \cite{antunes2008stochastic}
but we present it below in our notation for completeness.
Jobs can enter or leave the system or switch nodes which means that there are three transition types for each class of job. Thus the set $\mathbf{K}$ of different jump types can be represented as $\mathbf{K}=\{E^i,L^i,C^i:i=1,2\}$ where $n_{E^i}=n_{L^i}=1$ and $n_{C^i}=2$ for $i=1,2$. Let for $(j,i)\in\XX,\ \hat{e}_{j,i}=(\del_{(j,i),(k,\ell)})_{(k,\ell)\in\XX}$ be the $d$-dimensional vector which is 1 for entry $(j,i)$ and 0 for all other entries. 
The sets corresponding to the possible jumps of each type are
\begin{align*}
\Del^{E^1}
&= \{(\hat{e}_{j,i},\hat{e}_{j+1,i}):(j,i)\in S^{E^1}\},\quad
\Del^{E^2}
= \{(\hat{e}_{j,i},\hat{e}_{j,i+1}):(j,i)\in S^{E^2}\}\\
\Del^{L^1}
&= \{(\hat{e}_{j,i},\hat{e}_{j-1,i}):(j,i)\in S^{L^1}\},\quad
\Del^{L^2}
= \{(\hat{e}_{j,i},\hat{e}_{j,i-1}):(j,i)\in S^{L^2}\}\\
\Del^{C^1}
&= \Del^{L^1}\cup\{(\hat{e}_{j,i}+\hat{e}_{j',i'},\hat{e}_{j-1,i}+\hat{e}_{j'+1,i'}):(j,i,j',i')\in S^{C^1}\}\\
\Del^{C^2}
&= \Del^{L^2}\cup\{(\hat{e}_{j,i}+\hat{e}_{j',i'},\hat{e}_{j,i-1}+\hat{e}_{j',i'+1}):(j,i,j',i')\in S^{C^2}\}.
\end{align*}
where $S^{E^1}=\{(j,i)\in\XX:(j+1,i)\in\XX\}$ and $S^{E^2},\ S^{L^1},\ S^{L^2},\ S^{C^1},\ S^{C^2}$ are defined similarly.

Let $r\in\cls_N$. The rate of jumps corresponding to a job arriving at a node with $j$ and $i$ jobs of classes 1 and 2, respectively, is equal to the number of nodes in this configuration multiplied by the rate at which jobs enter the system. Namely, the rate $\bar{\Gam}^k_N(r,\nu)$ when $\nu=(\hat{e}_{j,i},\hat{e}_{j+1,i})\in\Del^{k}$ and $k=E^1$ is $Nr^{j,i}\times\lambda_1$, and similarly $\bar{\Gam}^k_N(r,\nu)=Nr^{j,i}\times\lambda_2,\ \nu=(\hat{e}_{j,i},\hat{e}_{j,i+1})\in\Del^k,\ k=E^2$.  
The rate of departures is given similarly but, since all jobs are processed simultaneously, we need to multiply the processing rate by the number of jobs at a given node. Specifically,  $\bar{\Gam}^k_N(r,\nu)=j\times Nr^{j,i}\times\tau_1$ for $\nu=(\hat{e}_{j,i},\hat{e}_{j-1,i})\in\Del^k,\ k=L^1$ and $\bar{\Gam}^k_N(r,\nu)= i\times Nr^{j,i}\times\tau_2$ for $\nu=(\hat{e}_{j,i},\hat{e}_{j,i-1})\in\Del^k,\ k=L^2$.
When jobs attempt to change nodes there are two possible outcomes (successful and unsuccessful switching) which we will consider separately. The case in which a job successfully switches nodes is analogous to a job leaving the system but rates are multiplied by the proportion of nodes in the configuration to which the job is switching. Thus for a job switching from a node with $j$ and $i$ jobs to a node with $j'$ and $i'$ jobs (of types 1 and 2, respectively) we have $\bar{\Gam}^k_N(r,\nu)= j\times Nr^{j,i}\times\gamma_1\times\frac{Nr^{j',i'}}{N-1}$ where $\nu=(\hat{e}_{j,i}+\hat{e}_{j',i'},\hat{e}_{j-1,i}+\hat{e}_{j'+1,i'})\in\Del^k,\ k=C^1$ and $\bar{\Gam}^k_N(r,\nu)= i\times Nr^{j,i}\times\gamma_2\times\frac{Nr^{j',i'}}{N-1}$ for $\nu=(\hat{e}_{j,i}+\hat{e}_{j',i'},\hat{e}_{j,i-1}+\hat{e}_{j',i'+1})\in\Del^k,\ k=C^2$.
Next consider unsuccessful switches. Recall that if a job attempts to switch to a node at which there is not enough room, then the job is rejected from the system. The rate at which such jumps occur is, again, analogous to the previous scenario except we instead multiply by the proportion of nodes without enough room for the job attempting to move. Let $r^C_i$ be the proportion of nodes without enough room to accommodate a job of type $i$ (i.e. nodes in states $(i',j')$ with $(j'A_1+i'A_2+A_i> C)$). Then $\bar{\Gam}^k_N(r,\nu)= j\times Nr^{j,i}\times\gamma_1\times\frac{Nr^C_1}{N-1}$ for $\nu=(\hat{e}_{j,i},\hat{e}_{j-1,i})\in\Del^k,\ k=C^1$ and $\bar{\Gam}^k_N(r,\nu)= i\times Nr^{j,i}\times\gamma_2\times\frac{Nr^C_2}{N-1}$ for $\nu=(\hat{e}_{j,i},\hat{e}_{j,i-1})\in\Del^k,\ k=C^2$.

With the above definition of $\bar{\Gam}^k_N$, the generator of $\{\mu_N(t)\}$ is as given by \eqref{eqn:gen}. $\Gam^k$ is defined to be the limit of $\bar{\Gam}^k_N$ which is simply given as
\begin{align}\label{eqn:39}
\Gam^k(r,\nu)
&= \left\{\begin{array}{cl}
j\times r^{j,i}\times\gamma_1\times r^{j',i'} & \text{for }\nu=(\hat{e}_{j,i}+\hat{e}_{j',i'},\hat{e}_{j-1,i}+\hat{e}_{j'+1,i'})\in\Del^k,\ k=C^1\\
i\times r^{j,i}\times\gamma_2\times r^{j',i'} & \text{for }\nu=(\hat{e}_{j,i}+\hat{e}_{j',i'},\hat{e}_{j,i-1}+\hat{e}_{j',i'+1})\in\Del^k,\ k=C^2\\
j\times r^{j,i}\times\gamma_1\times r^C_1 & \text{for }\nu=(\hat{e}_{j,i},\hat{e}_{j-1,i})\in\Del^k,\ k=C^1\\
i\times r^{j,i}\times\gamma_2\times r^C_2 & \text{for }\nu=(\hat{e}_{j,i},\hat{e}_{j,i-1})\in\Del^k,\ k=C^2\\
\bar{\Gam}^k_1(r,\nu) & \text{otherwise}
\end{array}
\right.
\end{align}
for $r\in\cls$. Clearly $\Gam^k(\cdot,\nu)$ is Lipschitz for all $k\in\mathbf{K},\nu\in\Del^k$ and \eqref{eqn:ratecon} is satisfied so Condition \ref{con:1} holds in this example. From Proposition \ref{prop:2} we then have that $\mu_N(t)\to\mu(t)$ uniformly on $[0,T]$ where $\dot{\mu}(t)=F(\mu(t))$ and $F$ is as in \eqref{eqn:8}, with $\Gam^k$ as defined above.

Now suppose that the arrival rates $\lambda_i,\ i=1,2$ can be modulated by exercising an additive control with values in $\frac{1}{\sqrt{N}}[-D,D],\ D<\iy,\ i=1,2$. One can also consider control of any of the other parameters $\{\tau_i,\ \gamma_i:i=1,2\}$ but for simplicity we will only consider the control of the arrival rates. Let 
\begin{align}\label{eqn:conset}
\begin{split}
\Lambda=\{u\in\RR^{\mathbf{\ell}_1}\times \{0\}^{\mathbf{\ell}-\mathbf{\ell}_1}&|u_j=u_1^*\in[-D,D],j=1,\ldots,|\Del^{E^1}|,\\ 
&\qquad u_k=u_2^*\in[-D,D],k=|\Del^{E^1}|+1,\ldots,|\Del^{E^2}|\}
\end{split}
\end{align}
where $\mathbf{\ell}=\sum_{i=1}^2\left(|\Del^{E^i}|+|\Del^{L^i}|+|\Del^{C^i}|\right)$ and $\mathbf{\ell}_1=\sum_{i=1}^2|\Del^{E^i}|$. The controls will take values in $\Lambda_N=\frac{1}{\sqrt{N}}\Lambda$ . For a $u\in\Lambda$ or $\Lambda_N$ let $u_1^*$ refer to the value of the first $|\Del^{E_1}|$ coordinates and $u_2^*$ refer to the value of the next $|\Del^{E^2}|$ coordinates. Define the controlled rate function as
\begin{align}\label{eqn:38}
\Gam_N^k(r,u,\nu)
=\left\{\begin{array}{cl}
Nr^{j,i}\times(\lambda_1+u_1^*) & \text{for }k=E^1,\ \nu=(\hat{e}_{j,i},\hat{e}_{j+1,i})\in\Del^{E^1}\\
Nr^{j,i}\times(\lambda_2+u_2^*) & \text{for }k=E^2,\ \nu=(\hat{e}_{j,i},\hat{e}_{j,i+1})\in\Del^{E^2}\\
\bar{\Gam}_N^k(r,\nu) & \text{otherwise},
\end{array}
\right.
\end{align}
where $u\in\Lambda_N$. Since controls in $\Lambda_N$ are $\clo\left(\frac{1}{\sqrt{N}}\right)$, Condition \ref{con:3} is easily seen to be satisfied for the example. 

From our assumption that $x_0^{j,i}>0$ for all $\ (j,i)\in\XX$, it follows that $\mu_t^{j,i}>0$ for all $(j,i)\in\XX$ and $0\leq t\leq T$. Using this and the form of $\Gam^k$ given in \eqref{eqn:39}, it is then easy to check that Condition \ref{con:8} is satisfied. Similarly our assumption on the initial conditions in Theorem \ref{thm:main} is satisfied as well. Recalling the definitions of $\Gam_N^k$ and $\Gam^k$ in \eqref{eqn:38} and \eqref{eqn:39}, respectively, we see that there exists a $\kappa \in (0,\infty)$ such that for all $y\in B(2\sqrt{N}), u\in\Lambda_N, \xi\in\cls_N(y)$
\begin{align*}
\sqrt{N}\left(\frac{1}{N}\Gam_N^{k}\left(\frac{1}{\sqrt{N}}y+\xi,u,\nu\right)-\Gam^{k}\left(\xi,\nu\right)\right)
\leq \kappa(1+\|y\|)
\end{align*}
and therefore Condition \ref{con:6} is satisfied. For $k\in\mathbf{K},\nu\in\Del^k$ define $h_1^k(\nu,\cdot):\cls\to\RR$ as
\begin{align*}
h_1^k(\nu,r)
=\left\{\begin{array}{cl}
r^{j,i} & \text{for }k=E^1,\ \nu=(\hat{e}_{j,i},\hat{e}_{j+1,i})\in\Del^{E^1}\\
r^{j,i} & \text{for }k=E^2,\ \nu=(\hat{e}_{j,i},\hat{e}_{j,i+1})\in\Del^{E^2}\\
0 & \text{otherwise}
\end{array}
\right.
\end{align*}
and $h_2^k(\nu,\cdot)$ as
\begin{align*}
\left\{\begin{array}{cl}
\lambda_1\times e_{j,i} & \text{for }k=E^1,\ \nu=(\hat{e}_{j,i},\hat{e}_{j+1,i})\in\Del^{E^1}\\
\lambda_2\times e_{j,i} & \text{for }k=E^2,\ \nu=(\hat{e}_{j,i},\hat{e}_{j,i+1})\in\Del^{E^2}\\
j\times\mu_1\times e_{j,i} & \text{for }k=L^1,\ \nu=(\hat{e}_{j,i},\hat{e}_{j-1,i})\in\Del^{L^1}\\
i\times\mu_2\times e_{j,i} & \text{for }k=L^2,\ \nu=(\hat{e}_{j,i},\hat{e}_{j,i-1})\in\Del^{L^2}\\
j\times \gamma_1\times (r^{j,i}\times e_{j',i'}+r^{j',i'}\times e_{j,i}) & \text{for }\nu=(\hat{e}_{j,i}+\hat{e}_{j',i'},\hat{e}_{j-1,i}+\hat{e}_{j'+1,i'})\in\Del^k,\ k=C^1\\
i\times \gamma_2\times (r^{j,i}\times e_{j',i'}+r^{j',i'}\times e_{j,i}) & \text{for }\nu=(\hat{e}_{j,i}+\hat{e}_{j',i'},\hat{e}_{j,i-1}+\hat{e}_{j',i'+1})\in\Del^k,\ k=C^2\\
j\times \gamma_1\times (r^{j,i}\times e_C^1+r^C_1\times e_{j,i}) & \text{for }\nu=(\hat{e}_{j,i},\hat{e}_{j-1,i})\in\Del^k,\ k=C^1\\
i\times \gamma_2\times (r^{j,i}\times e_C^2+r^C_2\times e_{j,i}) & \text{for }\nu=(\hat{e}_{j,i},\hat{e}_{j,i-1})\in\Del^k,\ k=C^2.\\
\end{array}
\right.
\end{align*}
Defining $H^k,\beta_k^N$ as in Condition \ref{con:7} with $h_1^k$ and $h_2^k$ we see that \eqref{eqn:betazero} is satisfied and thus Condition \ref{con:7} holds for the example.

We now introduce the following finite time horizon cost
\begin{align}\label{eqn:LQRCost}
J^N(U^N,v_N)=\E\int_0^T (\|V_N(t)\|^2+\alpha\|\sqrt{N}U^N(t)\|^2)dt,\ U^N\in\cla_N,
\end{align}
where $\alpha \in (0,\infty)$.
The cost function penalizes both the deviation from the nominal behavior and exercising rate control.
Note that this cost function satisfies the condition introduced below \eqref{eqn:costprelim}.  We have thus verified all the conditions needed for Theorem \ref{thm:main} and from this result it follows that a near optimal continuous feedback control for the diffusion control problem can be used to construct an asymptotically optimal sequence of control policies for this system. The diffusion control problem here takes the same form as \eqref{eqn:costdiff} with $\eta$ and $\beta$ as in \eqref{eqn:16} and $\sigma$ as in \eqref{eqn:7} with cost given as
\begin{align}\label{eqn:exdiffcost}
J(U,v_0)=\E\int_0^T (\|V(t)\|^2+\alpha\|U(t)\|^2)dt,\ U\in\cla(\Xi).
\end{align}
This is the classical stochastic linear-quadratic regulator problem which has been well studied (cf. \cite{fleming1976deterministic}). 
Replacing $[-D,D]$ with $\RR$ in the definition of the control set in \eqref{eqn:conset}, the optimal control for the limit stochastic LQR is given in feedback form as follows
\begin{align*}
u^*(s,y)= -B'(s)K^*(s)V(s)
\end{align*}
where $B$ is defined in terms of $\{h_1^k,k\in\mathbf{K}\}$ via the relation $\eta(t,u)=B(t)u$ and $K^*$ solves an appropriate Riccati equation (see \cite{fleming1976deterministic}). For implementing this feedback control for the prelimit system we truncate $u^*$ suitably; such a modification, in practice, has little to no effect for large $N$. We construct $U^N_g$ as in Section \ref{MainRes}, by taking $U^N_g(t)=\sqrt{N}u^*(t,V_N(t))$. 

We now present our numerical results. The above control policy was implemented (for $\alpha=.01$ and $.001$) on $N_{\text{trials}}=128$ different realizations of the stochastic process with the following parameters $N=10,000,T=10,C=6,A_1=1,A_2=1,\lambda_1=1,\lambda_2=1,\tau_1=1,\tau_2=1,\gamma_1=1,\gamma_2=1$.  We also simulate $128$ realizations of the corresponding uncontrolled system. Table \ref{tab:cost} shows the averaged cost over the 128 simulations for the controlled and uncontrolled systems. The control policy based on the optimal feedback control for the stochastic LQR leads to a reduction in cost of 12.7\% for $\alpha=.01$ and 15.5\% for $\alpha=.001$. 
\begin{table}[h]
\caption{Cost over 128 Simulations} \label{tab:cost}
\begin{tabular}{|c|c|c|c|}
\hline
 & Uncontrolled & Controlled with $\alpha=.01$ & Controlled with $\alpha=.001$\\
 \hline
Deviation Cost & 8.9556 & 8.1271 & 7.5649\\
\hline
Control Cost & 0 & $.01\times 25.37$ & $.001\times 256.8$\\
\hline
Total Cost & 8.9556 & 8.3809 & 7.8217\\
\hline
\end{tabular}
\end{table}
The deviations from the nominal values under the controlled and uncontrolled systems are computed by calculating the average,
\begin{align*}
\frac{1}{N_{\text{trials}}}\sum_{i=1}^{N_{\text{trials}}}\int_0^T\|V_N(s)\|^2ds
\end{align*}
for the two systems and the cost of exercising control is computed by the average,
\begin{align*}
\alpha\times\frac{1}{N_{\text{trials}}}\sum_{i=1}^{N_{\text{trials}}}\int_0^T\|\sqrt{N}U^N(t)\|^2ds.
\end{align*}
The deviations are smaller for the controlled system as expected. In general, one can achieve higher reduction in such deviations by decreasing the parameter $\alpha$ in the cost function. In practice the tuning parameter $\alpha$ suitably balances the cost of deviating from the nominal values and the cost for exercising control.

\setcounter{equation}{0}
\appendix

\numberwithin{equation}{section}

\section{Auxiliary Results}
\subsection{Conditions [A] and [T$_1$] of \cite{joffe1986weak}}
For the sake of the reader's convenience we present Theorem 2.3.2 and Conditions [A] and [T$_1$] of \cite{joffe1986weak} in a form that are used here. Let $\{M^n\}$ be a sequence of $\RR^k$-valued processes which are RCLL (right continuous with left limit) locally square-integrable martingales, defined on their own filtered probability space $\{(\Om^n, \clf^n,(\clf_t^n),\PP^n)\}$. Consider the following two conditions for a sequence of $k$-dimensional RCLL processes $\{X_N\}$, with $X_N$ defined on $(\Om^n,\clf^n,(\clf^n_t),\PP^n)$.
\begin{description}
\item{[A]} For each $\eps>0,\eta>0$ there exists a $\del>0$ and $n_0\in \NN$ with the property that for every family of stopping times $\{\tau_n\}_{n\in\NN}$ ($\tau_n$ being an $\clf^n$-stopping time on $\Om^n$) with $\tau_n\leq T-\delta$,
\begin{align*}
\sup_{n\geq n_0}\sup_{\theta\leq \del}P^n\{\|X^n_{\tau_n}-X^n_{\tau_n+\theta}\|\geq \eta\}\leq \eps.
\end{align*}
\item{[T$_1$]} For every $t$ in some dense subset  of $[0,T]$, $\{X_t^n\}_{n\in\NN}$ is a tight sequence of $\RR^k$ valued random variables.
\end{description}
\begin{theorem}[2.3.2 of \cite{joffe1986weak} (Rebolledo)]\label{thm:apptight}  Let $\lan M^n\ran\doteq\sum_{i=1}^k\lan M_i^n,M_i^n\ran$ be the predictable quadratic variation process associated with the $k$-dimensional local martingale $M^n$. Then if the sequence $\{\lan M^n\ran\}_{n\in\NN}$ of $\RR$-valued stochastic processes satisfies condition [A], the same condition holds for the sequence $\{M^n\}_{n\in\NN}$ of $\RR^k$-valued stochastic processes. Futhermore if $\{\lan M^n\ran\}_{n\in\NN}$ satisfies [T$_1$] then the same condition holds for $\{M^n\}_{n\in\NN}$. In particular if $\{\lan M^n\ran\}_{n\in\NN}$ satisfies [A] and [T$_1$], the sequence $\{\{\lan M^n_i,M^n_i\ran,i=1,\ldots,k\}\}_{n\in\NN}$ and $\{M^n\}_{n\in\NN}$ are tight in $\cld([0,T]:\RR^k)$.
\end{theorem}

\bigskip
{\bf Acknowledgements.}  Research supported in part by the National Science Foundation (DMS-1016441, DMS-1305120), the Army Research Office (W911NF-14-1-0331) and DARPA (W911NF-15-2-0122).


{\sc
\bigskip
\noi
A. Budhiraja (email: budhiraj@email.unc.edu)\\
E. Friedlander(email: ebf2@live.unc.edu)\\
Department of Statistics and Operations Research\\
University of North Carolina\\
Chapel Hill, NC 27599, USA

}

\end{document}